\documentclass[preprint]{ejpecp}
\usepackage{enumerate}
\usepackage{mathdots}
\usepackage{varioref}
\usepackage{hyperref}
\usepackage[capitalize,noabbrev]{cleveref}
\crefname{ineq}{inequality}{inequalities}
\creflabelformat{ineq}{#2{\upshape(#1)}#3} 
\usepackage[caption=false,font=footnotesize]{subfig}
\usepackage{tikz}
\usepackage{tikz-3dplot}
\usetikzlibrary{graphs,positioning,calc}
\tikzset{node distance=3mm,
      circled/.style={
        minimum size=40pt,
        circle,
        scale=0.8,
        draw=black
        },
      invis/.style={
        circle,
        scale=0.8,
        draw=white,
        color=white,
        opacity=0
        }
}

\SHORTTITLE{Entropic Symmetrization Resistance}
\TITLE{Entropic Symmetrization Resistance}

\AUTHORS{%
  Mokshay~Madiman\footnote{University of Delaware, United States of America.
    \EMAIL{madiman@udel.edu} \url{https://mokshaymadiman.wordpress.com}}
  \and %% remove this line and below if single author
  Emma~Pollard\footnote{Centro de Investigación en Matemáticas, México.
    \EMAIL{emma.pollard@cimat.mx} %\url{https://sites.google.com/view/emmapollard}
    }}

\KEYWORDS{probability ; symmetrization ; Bernoulli ; entropy ; variance} % Separate items with ;

\AMSSUBJ{60C05} 

\SUBMITTED{February 29, 2098} % Edit.
\ACCEPTED{February 29, 2099} % Edit.

\VOLUME{0}
\YEAR{2024} 
\PAPERNUM{0} 
\DOI{10.1214/YY-TN} 

\ABSTRACT{An asymmetric random variable $X$ in the reals is said to be variance symmetrization resistant if every independent random variable $Y$ in the reals that produces a symmetric sum $X+Y$ has a greater variance than that of $X$. 
Asymmetric Bernoulli random variables were shown to be variance symmetrization resistant by Kagan, Mallows, Shepp, Vanderbei, and Vardi (1999); Pal (2008) gave a proof using stochastic calculus. 
We introduce the notion of entropic symmetrization resistance on locally compact groups-- this means that the entropy of any independent symmetrizer $Y$ must exceed that of $X$. 
We show that asymmetric Bernoulli random variables exhibit entropic symmetrization resistance, 
and show a multidimensional generalization. 
We also explore basic aspects of the entropic symmetrization resistance problem in compact groups. In particular, we show that any distribution on a finite group that is entropic symmetrization resistant must lie on the boundary of the probability simplex, and describe precisely the class of all entropic symmetrization resistant distributions on $\mathbb{Z}_3$ and $\mathbb{Z}_4$.
}

\DeclarePairedDelimiter\abs{\lvert}{\rvert}%
\DeclarePairedDelimiter\ang{\langle}{\rangle}%
\DeclarePairedDelimiter\norm{\lVert}{\rVert}%
\DeclarePairedDelimiter{\ceil}{\lceil}{\rceil}%
\DeclarePairedDelimiter{\floor}{\lfloor}{\rfloor}%
\DeclarePairedDelimiter{\bbrack}{\llbracket}{\rrbracket}%
\makeatletter
\let\oldabs\abs
\def\abs{\@ifstar{\oldabs}{\oldabs*}}
\let\oldang\ang
\def\ang{\@ifstar{\oldang}{\oldang*}}
\let\oldnorm\norm
\def\norm{\@ifstar{\oldnorm}{\oldnorm*}}
\let\oldceil\ceil
\def\ceil{\@ifstar{\oldceil}{\oldceil*}}
\let\oldfloor\floor
\def\floor{\@ifstar{\oldfloor}{\oldfloor*}}
\let\oldbbrack\bbrack
\def\bbrack{\@ifstar{\oldbbrack}{\oldbbrack*}}
\makeatother
\DeclarePairedDelimiterX{\infdivx}[2]{(}{)}{%
  #1\;\delimsize\|\;#2%
}
\newcommand{\DKL}{D_{KL}\infdivx}

\DeclareMathOperator{\supp}{{\rm supp}}
\DeclareMathOperator{\Uniform}{Uniform}
\DeclareMathOperator{\Bern}{Bernoulli}
\DeclareMathOperator{\Bernoulli}{Bernoulli}

\DeclareMathOperator{\Poisson}{Poisson}
\DeclareMathOperator{\Var}{Var}
\DeclareMathOperator{\Cov}{Cov}

\DeclareMathOperator{\conv}{conv}
\newcommand{\calY}{\Sym}
\DeclareMathOperator{\Sym}{Sym}
\DeclareMathOperator{\op}{op}
\newcommand{\restr}[2]{#1{\mathchoice{\Big\rvert}{\big\rvert}{\rvert}{\rvert}}_{#2}} 
\newcommand{\Exp}{\operatorname*{Exp}}
\newcommand{\RL}{\mathbb{R}}
\renewcommand{\choose}{\binom}

\newcommand{\calB}{\mathcal{B}}

\newcommand{\calP}{\mathcal{P}}
\newcommand{\ra}{\rightarrow}
\newcommand{\hatf}{\hat{f}}
\newcommand{\aff}{{\rm aff}}
\newcommand{\tr}{{\rm tr}} %trace
\newcommand{\Ex}{{\rm Ex}} %extreme points
\newcommand{\R}{\mathbb{R}}
\newcommand{\Z}{\mathbb{Z}}
\newcommand{\convolve}{*}

\begin{document}
%\tableofcontents
\section{Introduction}

Quantitative measures of symmetry of various kinds of objects appear in many parts of mathematics. If one is dealing with a real-valued function $f:G\ra \RL$ on a group $(G,+)$, and one wishes to measure how symmetric $f$ is, one may consider the splitting $f=f_s + f_a$ into symmetric and anti-symmetric parts by setting $f_s(x)=\frac{1}{2} [f(x)+f(-x)], f_a(x)=\frac{1}{2} [f(x)-f(-x)]$, and then measure the size of the anti-symmetric part using some norm on (an appropriate subset of) the functions on $G$. As another example, the book \cite{Toth15:book} is a study of various quantitative measures of symmetry for convex sets in a Euclidean space.
Our focus in this paper is on some intriguing measures of symmetry for probability measures on a locally compact group, two of which we introduce in this paper and wish to advertise as objects deserving of further study.
We focus on developing some basic results on quantitative asymmetry in two cases: 
Euclidean spaces $\RL^n$ (with particular emphasis on $n=1$), and  compact groups (with particular emphasis on finite abelian groups). 

Our investigation was inspired by work of Kagan, Mallows, Shepp, Vanderbei and Vardi \cite{KMSVV99}, who showed that asymmetric Bernoulli random variables exhibit a phenomenon of \emph{symmetrization resistance}, as measured by the variance. Specifically, they defined a {\it symmetrizer} of $X$ to be any random variable $Y$ independent of $X$ and such that $X+Y$ has a symmetric distribution about zero, and 
showed that if $X$ is a Bernoulli random variable with  parameter $p\neq\frac{1}{2}$, then the variance of any symmetrizer of $X$ is at least that of $X$ itself. Since $Y=-X'$ where $X'$ is an independent copy of $X$ is always a symmetrizer, their result has the implication that one cannot do better when looking for symmetrizers than to use this obvious one. As observed by \cite{KMSVV99}, symmetrization resistance is surprisingly difficult to prove even for the simplest non-trivial examples. Perhaps as a consequence, there has been barely any research on this phenomenon in the intervening decades; the only relevant work (to our knowledge) is a paper of Soumik Pal \cite{Pal08}, who gave a second proof of the symmetrization resistance of the asymmetric Bernoulli distribution using stochastic calculus. 

While variance is a satisfactory way to quantify the ``spread'' of a probability distribution on the real line, 
we wish to consider other notions of spread as well that are meaningful for more general spaces. 
With this in mind, we make the following definitions. Let $\mathcal{P}(G)$ be the set of probability measures on (the Borel $\sigma$-field of) a topological space $G$, and $\calP$ 
be a convex subset of $\mathcal{P}(G)$.

\begin{definition}
A {\it spread functional} on $\calP$
is a quasiconcave, upper semicontinuous function $\Phi:\calP\ra\bar{\RL}$. 
\end{definition}

Let $\delta_x$ denote the Dirac measure or point mass at $x$ defined by $\delta_x(A)=1$ if $x\in A$, and $\delta_x(A)=0$ otherwise.
We say $\calP$ {\it contains point masses} if it contains the Dirac measures $\delta_x$ for each $x\in G$. Clearly, if $\calP$ contains point masses, $\calP$ contains the set $\calP_f(G)$ of all probability measures on $G$ with finite support.
Motivated by the fact that the variance is 0 for any Dirac measure on the real line, we make the following definition.

\begin{definition}
A spread functional $\Phi:\calP\ra\bar{\RL}$ whose domain contains point masses is {\it grounded} when $\Phi(\mu)=0$ if and only if $\mu$ is a point mass. 
\end{definition}

For a grounded spread function, the quasiconcavity property
$\Phi(\lambda \mu +(1-\lambda)\nu) \geq \max\{\Phi(\mu), \Phi(\nu)\}$ implies that $\Phi(\mu)$ is nonnegative for any probability measure $\mu$ with finite support. Moreover, the upper semicontinuity of $\Phi$ (with respect to convergence in distribution) implies that if $G$ is a Polish space (so that every probability measure on $G$ can be approximated by finitely supported ones), then $\Phi$ is nonnegative on  $\calP$.

For $G=\RL$, the variance $V(\mu)=\int_{\RL} x^2 d\mu -(\int _{\RL} x d\mu)^2$ is a linear spread functional on $\calP(\RL)$; however, a disadvantage of the variance is that it is not obvious how to extend the definition to more general spaces $G$. At the cost of restricting attention to the set $\calP_c(G)$ of probability measures on $G$ with countable support, the Shannon entropy applies to arbitrary sets $G$; indeed, if $A$ is the support of $\mu\in \calP_c(G)$, then the entropy 
$$
H(\mu)=-\sum_{x\in A} p(x)\log p(x),
$$
where $p$ is the probability mass function of $\mu$ on $A$. One may extend the entropy to a spread functional on $\calP(G)$ by setting $H(\mu)=\infty$ when $\mu\notin \calP_c(G)$. 
More generally, one can also consider the R\'enyi entropies of arbitrary order (although we do not do so in this paper).

Let $(G,+)$ be a locally compact group (written additively even if $G$ is nonabelian) equipped with its Borel $\sigma$-field $\calB(G)$. Let $*$ denote the corresponding convolution operation on $\calP(G)$, i.e., $\mu*\nu$ is the pushforward of the probability measure $\mu\otimes \nu$ under the function $(x,y)\mapsto x+y$.

\begin{definition}
For convex $\calP\subset \calP(G)$ and $\mu\in \calP$, the set $\Sym_{\calP}(\mu)$ of symmetrizers for $\mu$ is defined by 
$$
\Sym_{\calP}(\mu)=\{\nu\in\calP:\mu*\nu(B)=\mu*\nu(-B) \forall B\in \calB(G)\} .
$$
The {\em $\Phi$-asymmetry} of $\mu$ with respect to $\calP$ is defined by
$$
A_{\Phi, \calP}(\mu) := \inf_{\{\nu\in \Sym_{\calP}(\mu)\}} \Phi(\nu) . 
$$
\end{definition}

If $\Phi$ is a grounded spread function, it is easy to see that $A_\Phi(\mu)=0$ if and only if $\mu$ itself is symmetric up to a shift; so this is a meaningful measure of asymmetry.

Let $X$ be a $G$-valued random variable on some probability space. It is a common practice to identify a random variable with its distribution and write $\Phi(X)$ for $\Phi(P_X)$ where $P_X$ is the distribution of $X$; this is routinely done for both the variance $V(X)$ and the Shannon entropy $H(X)$. By adopting this practice, we may rewrite the definition above as follows:
$$
A_{\Phi, \calP}(X) := \inf_{\{Y\in Sym_{\calP}(X)\}} \Phi(Y) ,
$$
where $\Sym_{\calP}(X)$ is the set of random variables $Y$ with distributions in ${\calP}$, such that $X, Y$ are independent and $X+Y$ has a symmetric distribution.

For locally compact groups that are not discrete, it is of particular interest to consider the  
set $\calP_c(G)$ of all probability measures on $G$ that are absolutely continuous with respect to the (left) Haar measure $dx$. Of course, in this case, we can identify any element of $\calP_c(G)$ by its density $f:G\ra [0,\infty)$ with respect to Haar measure. Then it is natural to consider the {\it differential entropy} as a spread functional, defined by
$$
h(f)=-\int_G f(x)\log f(x) dx.
$$
This spread functional typically can take negative values. Nonetheless it is natural to ask questions about $h$-asymmetry on $\calP_c(G)$, since one can view them as questions about the related functional $N(f)=e^{h(f)}$, which is always nonnegative. 

\begin{remark}
If the Haar measure of $G$ is non-atomic, then $\calP_c(G)$ does not contain point masses. 
In this case, although $N(f)$ is nonnegative, the natural upper semicontinuous extension of $N$ from $\calP_c(G)$ to $\calP(G)$ is still not grounded, since we would set $h(\mu)=-\infty$ for any probability measure $\mu$ that is not absolutely continuous with respect to Haar measure, whence $N(\mu)=0$ for all $\mu\in \calP_f(G)$.

On the other hand, if $G$ is discrete, so that one may take the Haar measure to be counting measure, $\calP(G)=\calP_c(G)$ is the closure of $\calP_f(G)$, and the differential entropy $h$ and entropy $H$ coincide. In particular, $h$ is a grounded spread functional for discrete groups.
\end{remark}

\begin{definition}
A probability measure $\mu\in \calP(G)$ is {\em $\Phi$-symmetrization resistant relative to $\calP$ with constant $c>0$} if  
$$
A_{\Phi, \calP}(\mu) \geq c\Phi(\mu) .
$$
If this holds with $c=1$, we simply say that $\mu$ is {\it $\Phi$-symmetrization resistant} in $\calP$.
\end{definition}

In this language, the result of Kagan et al. \cite{KMSVV99} says that for any $p\neq \frac{1}{2}$, the Bernoulli distribution with parameter $p$ is $V$-symmetrization resistant in $\calP(\RL)$.

Our first main contribution in this paper is the following theorem (see Theorem~\ref{thm:HY_exceeds_HX}).

\begin{theorem}
Asymmetric Bernoulli distributions on $\RL$ are $H$-symmetrization resistant in $\calP(\RL)$.    
\end{theorem}
 
We also show higher-dimensional analogues of this fact. Along the way,  we also give a new proof of the fact that asymmetric Bernoullis are $V$-symmetrization resistant in $\calP(\RL)$ (and a multidimensional generalization). Our analysis is based on an  investigation of the space of symmetrizer PMFs $\calY$, culminating in \cref{thm:Bernoulli_calY_representation}
where we show that a certain set of simple symmetrizer PMFs forms a basis for the affine hull of $\calY$. 

Our second main contribution is to begin an investigation of entropic symmetrization resistance in the setting of compact groups.
A key observation is the following theorem (see
Theorem~\ref{thm:no_symmetrizers_in_interior}).

\begin{theorem}
On a compact group $G$, if $\mu$ is $h$-symmetrization resistance in $\calP_c(G)$, then the density $f$ of $\mu$ has infimum 0 on $G$. 
\end{theorem}

In particular, this means that for finite groups, any symmetrization-resistant distributions lie on the boundary of the probability simplex. We identify precisely the set of all symmetrization-resistant distributions on $\mathbb{Z}_3$ and $\mathbb{Z}_4$, the simplest nontrivial examples, and also make several observations about symmetrization resistance in larger finite abelian groups.

We also consider $V$-symmetrization resistance for non-Bernoulli probability measures on $\RL$. Indeed, we show (in Theorem~\ref{thm:V-symm-nonneg} by adapting Pal's method) that any probability measure on $\RL_+$ is $V$-symmetrization resistant with a constant $c$ that depends on its ``variance profile''. For example, we show that the exponential distribution is $V$-symmetrization resistant with constant $0.47$, and that the Poisson  distribution is $V$-symmetrization resistant with constant $0.26$.
When considering these problems, we conjectured that the exponential distribution is both $V$-symmetrization resistant and $h$-symmetrization resistant (with constant 1) in $\calP_c(\RL)$. Recently Jiange Li \cite{Li26} confirmed this conjecture after we communicated it to him.

This paper is organized as follows. 
In \cref{sec:general}, we  make some general observations on the symmetrization problem in Euclidean spaces. 
In \cref{sec:reals} we investigate random variables on the reals. In particular, we investigate the space of symmetrizer PMFs $\calY$ for asymmetric Bernoullis in \cref{sec:Bernoulli_representation}.
In Section~\ref{sec:var-reals}, we give a new proof of $V$-symmetrization resistant of asymmetric Bernoullis, and also show that Pal's method based on stochastic calculus quantifies $c$-symmetrization resistance (in terms of variance) for any probability measure on the positive real line.
In Section~\ref{sec:Bernoulli_entropy}, we show that asymmetric Bernoullis are entropic symmetrization resistant (\cref{thm:HY_exceeds_HX}), and consider both variance and entropic symmetrization resistance in higher dimensions in \cref{sec:hypercube},
showing the effects of coordinate dependencies in the higher dimensional setting. 
In \cref{sec:groups} we make some general observations about symmetrization in compact groups and resolve the symmetrization resistance problem in $\mathbb{Z}_3$ and also show several results in larger finite abelian groups.
Section~\ref{sec:rmk} makes some concluding remarks.

Some results of this article are distilled from the Ph.D. thesis \cite{Pol23:PhD}.

\section[Symmetrization in the reals]{Symmetrization in $\mathbb{R}^d$}\label{sec:reals}

\subsection{General observations}
\label{sec:general}

Certain general observations can be made about the symmetrization problem.
We first fix some notation.

We denote the set of integers from $1$ to $n$ by $[n]$.
For subsets $A$ and $B$ of $\mathbb{R}^d$ and $k \in \mathbb{R}$, we denote the Minkowski sum by $A+B$, and we write $kA = \{ ka \mid a \in A \}$. 
We denote the dimension of a subset $E$ of $\mathbb{R}^d$ by $\dim(E)$ and its affine hull by $\aff(E)$. 
We denote the indicator function of $E$ by $\mathds{1}_E$. 
By a slight abuse of notation we also write $\mathds{1}_{z}$ for $\mathds{1}_{\{z\}}$ when $z \in \mathbb{R}^d$. 

When $X$ is an $U$-valued random variable (r.v.) we write $X \in_R U$. 
The probability mass function (PMF) of a discrete r.v. $Z$ is written $f_Z$ and its support is $\supp(Z)$. 
In the following we will deal exclusively with discrete random variables with finite expectation; as such, we always assume the existence of PMFs. 

We denote the Shannon entropy of a discrete r.v. $Z$ (or, equivalently, of its PMF $f_Z$) by 
\[ H(Z) = H(f_Z) = \mathbb{E} (-\log f_Z) = -\sum_{z \in \supp(Z)} f_Z(z)\log(f_Z(z)),\]
where the logarithm is taken with base $2$.

\begin{definition}
We say that $Z \in_R \mathbb{R}^d$ is \emph{symmetric about zero} if for any measurable $E \subseteq \mathbb{R}^d$, $\mathbb{P}(Z \in E) = \mathbb{P}(Z \in -E)$. 
If $Z - w$ is symmetric about zero (for some fixed $w \in \mathbb{R}^d$) then we say $Z$ is \emph{symmetric about $w$}. 
We say that $Z$ is \emph{symmetric} to mean that it is symmetric about \emph{some} $w$; 
we say that $Z$ is \emph{asymmetric} if it is not symmetric about \emph{any} $w$. 
\end{definition}
For discrete $Z$, symmetry about zero is equivalent to the condition that for all $z \in \supp(Z)$, $f_Z(-z) = f_Z(z)$. 
We call the equations $f_Z(-z) = f_Z(z)$ the \emph{symmetry equations} of $Z$. 

Consider independent $X,Y \in_R \mathbb{R}^d$ such that the sum $X+Y$ is symmetric about zero. 
In this situation we say that $Y$ \emph{symmetrizes} (or is an  \emph{(independent) symmetrizer} of) $X$ (and vice-versa). 

\begin{definition}
Denote the space of symmetrizer distributions of $X \in_R \mathbb{R}^d$ by 
\[ \calY(X) = \calY(f_X) = \{ f_Y \mid Y \text{ is an independent symmetrizer of $X$} \}. \]
When the r.v. $X$ is clear from the context we will simply write $\calY$. 
\end{definition}

It is straightforward that the spaces $\calY$ are convex. 

\begin{definition}[\cite{KMSVV99}]
We say that $X \in_R \mathbb{R}^d$ is \emph{variance symmetrization resistant} with constant $c>0$ when every $f_Y \in \calY(X)$ satisfies $\Var(f_Y) \ge c\Var(f_X)$. 
\end{definition}

We introduce the following analogous notion: 
\begin{definition}
We say that $X \in_R \mathbb{R}^d$ is \emph{entropic symmetrization resistant} with constant $c>0$ when every $f_Y \in \calY(X)$ satisfies $H(f_Y) \ge cH(f_X)$. 
\end{definition}

When $d=1$ we are interested in the case $c=1$; when $c$ is unspecified, we mean $c=1$. 

We state formally the theorem of Kagan, Mallows, Shepp, Vanderbei and Vardi~\cite{KMSVV99}: 
\begin{theorem}
Asymmetric Bernoulli random variables are variance symmetrization resistant. 
\end{theorem}
We re-prove this result in \cref{sec:Bernoulli_variance} as \cref{thm:VarY_exceeds_VarX}.

\begin{remark}
If a nontrivial $X \in_R\mathbb{R}^d$ is symmetric about any $w \in \mathbb{R}^d$, then $\mathbb{E}(X-w) = 0$ and $w = \mathbb{E}X$; 
this shows that $\mathds{1}_{-w} \in \calY(X)$. 
Trivially, $\Var(\mathds{1}_{-w}) = H(\mathds{1}_{-w}) = 0$; therefore, symmetric $X$ are neither variance nor entropic symmetrization resistant with any constant $c>0$. 
\end{remark}

The following lemma, the proof of which is straightforward, introduces a class of symmetrizer PMFs that will be instrumental in our investigation. %proving \cref{thm:Bernoulli_calY_representation}. 
\begin{lemma}\label[lemma]{lem:hatf_are_symmetrizers}
Let $f$ be a PMF and let $z \in \mathbb{R}$. 
Then
\[
\hatf_z \coloneqq \frac{1}{2} \left( \mathds{1}_{-z} \convolve f_- + \mathds{1}_{z} \convolve f_- \right) \in \calY(f)
\]
where $f_-$ denotes the reversed PMF $f_-(w) = f(-w)$. 
\end{lemma}

In addition to convexity, the spaces $\calY$ exhibit a second useful closure property:
\begin{lemma}[Closure under subtraction]\label[lemma]{lem:closure_under_subtraction}
Let $f,g \in \calY$ with $\supp(g) \subseteq \supp(f)$. 
Let $C = \inf_{z \in \supp(g)} \frac{f(z)}{g(z)}$ and let $c \in [0,C]$. %we don't require $C > 0$. 
Then $h = \frac{f-cg}{1-c} \in \calY$. 

Moreover, $\supp(h) \subseteq \supp(f)$, and if $\abs{\supp(f)} < \infty$ then the inclusion is proper if and only if $c = C$. 
\begin{proof}
Let $f,g \in \calY$. 
First note that for $z \in \supp(g)$, $cg(z) \le Cg(z) \le f(z)$, so $f - cg \ge 0$ by construction. 
Also, $1-c = \sum_{z \in \supp(f)} (f(z) - cg(z))$, so $h$ is a PMF with $\supp(h) \subseteq \supp(f)$. 

Observe that for every $z \in \mathbb{R}$, both $\varphi = f$ and $\varphi = g$ solve the symmetry equation
\[ 
%q \varphi(-z+1) + p\varphi(-z-1) = 
(\phi \convolve f_X)(-z) = (\phi \convolve f_X)(z) 
%= q\varphi(z+1) + p \varphi(z-1) .
\]
%\begin{align*}
%q\varphi(2z+2) + p\varphi(2z) &= q\varphi(-2z) + p\varphi(-2z-2) 
%\shortintertext{and}
%q\varphi(2z) + p\varphi(2z-2) &= q\varphi(-2z+2) + p\varphi(-2z) .
%\end{align*}
This equation is linear in $\varphi$, so $\varphi = h$ is also a solution. % and $h \in \calY$. 
\end{proof}
\end{lemma}

The following easy observation will be useful.
\begin{lemma}[Affine invariance]\label[lemma]{lem:affine_invariance}
Let $X \in_R \mathbb{R}^d$ be discrete, 
let $f_Y \in \calY(X)$, 
let $A$ be a linear bijection on $\mathbb{R}^d$
and let $b \in \mathbb{R}^d$. 
Then $f_{AY - b} \in \calY(AX + b)$. 

Moreover, for any fixed $c>0$, $H(AY-b) \ge cH(AX + b)$ iff $H(Y) \ge cH(X)$, 
and for $d=1$ and nonzero $a \in \mathbb{R}$, $\Var(aY-b) \ge c\Var(aX+b)$ iff $\Var(Y) \ge c\Var(X)$. 

\begin{proof}
For the first claim it is sufficient to check the symmetry equations, 
and indeed
\[ 
f_{AX+b + AY-b}(z)
%= f_{A(X+Y)}(z)
= f_{X+Y}(A^{-1} z)
= f_{X+Y}(-A^{-1} z)
%= f_{A(X+Y)}(-z)
= f_{AX+b + AY-b}(-z). 
\]

Also, $H(AY-b) = H(Y)$ and $H(AX+b) = H(X)$. 
Finally, $\Var(aY-b)= a^2\Var(Y)$ and $\Var(aX+b) = a^2\Var(X)$. 
\end{proof}
\end{lemma}

The following special case is of particular utility. 
\begin{corollary}[Symmetry about zero]\label[corollary]{cor:psymmetry}
If $f_Y \in \calY(X)$, then $f_{-Y} \in \calY(-X)$.
\end{corollary}

For appropriately scaled and centered random variables, the following notation will be convenient. 
\begin{definition}\label[definition]{def:Sr}
For $r \in [0,1]$, write
$ S^r = 2\mathbb{Z} + \{\pm r\}, $
and
$ \calY^r 		= \{ f \in \calY 			\mid \supp(f) \subseteq S^r \} .$
\end{definition}
Note that the sets $S^r$ partition the real line. 
The spaces $\calY^r$ will turn out to partition a superset of the set of extreme points of $\calY$.

\subsection{Symmetric components}
\begin{definition}[\cite{KMSVV99}]
A PMF $f$ has a \emph{symmetric component} if there exist PMFs $g,h$ such that $g$ is symmetric, $\abs{\supp(g)} \ge 2$, and $f = g \convolve h$. 
\end{definition}

Demonstrating the existence of a symmetric component is a useful for showing that a random variable is \emph{not} symmetrization resistant. 
The following lemma is due to \cite{KMSVV99}; we extend it to the case of entropy. 
\begin{lemma}\label[lemma]{lem:symmetric_components}
If a PMF $f$ has a symmetric component, then it is neither variance symmetrization resistant nor entropic symmetrization resistant. 

\begin{proof}
Let $f = g \convolve h$ satisfy the assumption, where $g$ is symmetric and nondegenerate ($\abs{\supp(g)} \ge 2$). 
For a PMF $\varphi$ define $\varphi_-$ by $\varphi_-(z) = \varphi(-z)$ everywhere. 
Note that $h \convolve h_-$ is a symmetric PMF. 
Thus $f\convolve h_- = g \convolve (h \convolve h_-)$ is symmetric (the convolution of two symmetric PMFs is symmetric) 
and $h_- \in \calY(f)$. 
Since $g$ is nondegenerate, $h \ne f$ and $h_- \ne f_-$. 
Moreover, $\Var(h_-) = \Var(h)$ and $H(h_-) = H(h)$. 
Therefore we need only show that $\Var(h) \le \Var(f)$ and $H(h) \le H(f)$. 

%For the first result, we simply note that we may construct 
For variance, taking independent $X \sim g$ and $Y \sim h$, we have $X+Y \sim f$ and
\[ \Var(f) = \Var(X+Y) = \Var(X) + \Var(Y) = \Var(g) + \Var(h) > \Var(h) . \] 
Noting that $h$ cannot be periodic, %(else it would integrate to $\infty$)
we apply strict concavity and translation-invariance of entropy: 
\[ 
H(f) 
= H(g \convolve h) 
= H\bigg(\sum_{w \in \supp(g)} g(w)h(\cdot-w)\bigg) 
> \sum_{w \in \supp(g)} g(w) H(h(\cdot-w))
%= \sum_w g(w) H(h) 
= H(h). \qedhere
\] 
\end{proof}
\end{lemma}

\begin{remark}
An obvious question is whether asymmetry of $X$ must force $H(Y) \ge cH(X)$ or $\Var(Y) \ge c\Var(X)$ for some universal constant $c$. 
The above lemma shows that this is not the case. 
%Indeed, if we allow $X$ to have an asymmetric component, $H(Y)$ may be made to be arbitrarily small relative to $H(X)$, and $\Var(Y)$ may be made arbitrarily small relative to $\Var(X)$.

%To see this, 
Indeed, fix $c>0$ and let $U \sim \Bernoulli(p,\pm1)$ (with $p \ne \frac{1}{2}$) and $V \sim \Uniform([-2^{n},2^n] \cap \mathbb{Z} \setminus\{0\})$, be independent. 
Defining $X = U+V$, we have $f_{-U} \in \calY(X)$. 
Taking $n > 2$ to be large enough that $\frac{2}{n+1} < c$, 
\[ H(f_{-U}) < 1 = \frac{1}{n+1}H(V) < cH(V) < cH(X), \]
and
\[ \Var(f_{-U}) = 4p(1-p) < 2 < \frac{2}{n+1}\Var(V) < \frac{2}{n+1} \Var(X) < c\Var(X) ,\]
where in the latter computation we used the crude estimate 
\[ \Var(V) 
= \mathbb{E}V^2 
= 2\sum_{i \in [2^n]} \frac{i^2}{2^{n+1}}
= \frac{1}{2^n} \sum_{i \in [2^n]} i^2 
> 2^n
> n+1 .
\]
\end{remark}

\begin{remark}
One may also ask whether, given distributions $f_1, f_2, \dots$ which are already symmetrization resistant, how one may combine them to produce new symmetrization resistant distributions. 
Apart from tensorization with appropriate constraints (\cref{sec:hypercube}), we do not know of any way to do this. 
For example, though asymmetric Bernoullis are ($H-$ and $V$-) symmetrization resistant, binomials in general appear not to be (\cite{KMSVV99,Pol23:PhD}); thus convolution does not work. 
Similarly, the pointwise limit of symmetrization-resistant densities might not be symmetrization resistant.
For example, if $f_n \sim \Bernoulli(\frac{1}{2}+\frac{1}{n})$ we recover $f_n \rightarrow f \sim \Bernoulli(\frac{1}{2})$. 
Consider also that this example be used to construct less obvious examples.
For example, take $f_n$ as above and $g_n \rightarrow g$ (with $g_n$ and $g$ nontrivial).  
Then, even if $f_n \convolve g_n \rightarrow f \convolve g$ and $f_n \convolve g_n$ is symmetrization resistant for each $n$, the limit $f \convolve g$ is not symmetrization resistant by \cref{lem:symmetric_components}. 
\end{remark}

\subsection[An affine basis spanning the symmetrizer space for asymmetric Bernoulli]{An affine basis spanning the symmetrizer space for asymmetric Bernoullis}
\label[subsection]{sec:Bernoulli_representation}

In this subsection, we first provide 
a basis for $\aff(\calY^0)$, and then expand it to all of $\calY$.  
%in \cref{lem:calYzero_representation}

\begin{definition}
For integers $k \ge 1$, we denote 
\[ \hatf_k^0 = \hatf_{2k-1} = \frac{1}{2}\left(\mathds{1}_{-2k+1} \convolve f_{-X} + \mathds{1}_{2k-1}\convolve f_{-X}\right) . \]
\end{definition}

\begin{lemma}\label[lemma]{lem:calYzero_representation}
The set $\{ \hatf_i^0 \mid i \in \mathbb{N} = \{1, 2, \dots \}$ is a basis for $\aff(\calY^0)$. 
Moreover, the coefficients in an expansion $f = \sum_{i=1}^{\infty} \alpha_i^0 \hatf_i^0$ with respect to this basis
are given by the formula
$\alpha_i^0 = \frac{2(p f(-2i) - q f(2i))}{p-q}.$

\begin{proof}
Fix $f \in \calY^0$. 
Note that for all $z \not\in S^0 = 2\Z$ and positive integers $i$, $f(z) = \hatf_i^0(z) = 0$.

Solving the symmetry equation $pf(0) + qf(2) = pf(-2) + qf(0)$ for $f(0)$, we obtain
\(
f(0) = \frac{pf(-2) - qf(2)}{p-q},
\)
so that defining 
\(
\alpha_1^0= 2f(0) 
\)
results in 
\( f(0) 
= \alpha_1^0 \hatf_1^0(0) = \sum_{k \ge 1} \alpha_k^0 \hatf_k^0(0)\). 
Since 
$f(z)=0$ for all other $\abs{z} < 2$, this is the unique value of $\alpha_1^0$ satisfying the formula $f(z) = \sum_{k \ge 1} \alpha_k^0 \hatf_k^0(z)$ when $\abs{z} < 2$. 

We proceed by induction on $k$. 

Our inductive hypothesis is that 
\(
\alpha_i^0 = \frac{2(pf(-2i) - qf(2i))}{p-q}
\)
(where $i \in [k]$) is the unique solution to $f(z) = \sum_{i=1}^k \alpha_i^0 \hatf_i^0(z)$ for $\abs{z} < 2k$. 
We compute
\begin{align*}
f(2k) = \sum_{i \ge 1} \alpha_i^0 \hatf_i^0(2k) 
&= \alpha_{k}^0 \hatf_{k}^0(2k) + \alpha_{k+1}^0\hatf_{k+1}^0(2k) \\
&= \left(\frac{2(pf(-2k) - qf(2k))}{p-q}\right)\left(\frac{q}{2}\right) + \alpha_{k+1}^0 \frac{p}{2} 
\shortintertext{thus}
\alpha_{k+1}^0 &= \left(\frac{2}{p}\right)\left(f(2k) - \left(\frac{pqf(-2k) - q^2f(2k)}{p-q}\right)\right) \\
&= \frac{2((p-q+q^2)f(2k) - pqf(-2k))}{p(p-q)} \\
%&= \frac{2(p^2f(2k) - pqf(-2k))}{p(p-q)} \\
&= \frac{2(pf(2k) -qf(-2k))}{p-q} .
\end{align*}
where we have used the identity $p - q +q^2 %= p-(1-p)+(1-p)^2 
= p^2$. % in the last equality. 
%Then,}
%\alpha_{k+1}^0 &= \frac{2(pf(2k) -qf(-2k))}{p-q} .
%\intertext{
Note that the symmetry equation 
\[ pf(2k) +qf(2k+2) = pf(-2k-2)+qf(-2k) \]
may be arranged as 
\[pf(2k) - qf(-2k) = pf(-2(k+1)) - qf(2(k+1)) \]whence the desired formula %for $\alpha_{k+1}_0$. 
\(
\alpha_{k+1}^0 = \frac{2(pf(-2(k+1)) - qf(2(k+1)))}{p-q} 
\).
A similar argument shows that the same value of $\alpha_{k+1}^0$ is the unique solution to $f(-2k) = \sum_{i\ge1} \alpha_i^0 \hatf_i^0(-2k)$. 
Since $f(z) = \hatf_i^0(z) = 0$ for all $\abs{z} \in (2k,2(k+1))$,
%(and every $i$), we see that the coefficients $\alpha_j^0$ ($j \in [k+1]$) given according to the above formula constitute the unique solution to $f(u) = \sum_{i=0}^{k+1} \alpha_j^0 \hatf_j^0(u)$ when $\abs{u} < 2(k+1)$. 
%Conversely, every solution $f(z)$ to the symmetry equations with $\abs{z} < 2(k+1)$ is of this form. 
this completes the induction step; 
%
%Now let $z \in \mathbb{R}$ and and choose a positive integer $k$ such that $\abs{z} < 2k$. %e.g. $k = \ceil{\frac{1}{2}\abs{z}} + 1$. 
%Then %by the definition of the $\hatf_i$, 
%$\sum_{i \ge 1} \alpha_i^0 \hatf_i^0 (z) = \sum_{i=1}^k \alpha_i^0 \hatf_i^0(z) = f(z)$, so 
consequently $f = \sum_{k \ge 1} \alpha_i^0 \hatf_i^0$ everywhere on $\mathbb{R}$. 
Finally, note that 
\[
\sum_{i=1}^{\infty} \alpha_i^0 
= \sum_{i=1}^{\infty} \alpha_i^0 \sum_{k \in \mathbb{Z}} \hatf_i^0(2k) 
= \sum_{k \in \mathbb{Z}} \sum_{i=1}^{\infty} \alpha_i^0 \hatf_i^0(2k) 
= \sum_{k \in \mathbb{Z}} f(2k) = 1. \qedhere
\]
\end{proof}
\end{lemma}

\begin{definition}\label[definition]{def:Bern_Ir}
For $r \in [0,1]$ we define the \emph{index set $I^r$} to be 
\[ 
I^r = \begin{cases}
    \mathbb{N} = \{1, 2, \dots \} & \text{ when } r = 0 \\
    \mathbb{Z} & \text{ when } r \in (0,1) \\
    \mathbb{N}_0 = \{0, 1, 2, \dots \} & \text{ when } r = 1 .
    \end{cases}
\]
\end{definition}
 
\begin{definition}\label[definition]{def:Bern_hatfkr}
For $r \in [0,1]$ and $k \in I^r$ we define functions $\hatf_k^r$ by
\[ \hatf_k^r = \hatf_{2k-1+r} = \frac{1}{2}\left(\mathds{1}_{-2k+1-r}\convolve f_{-X} + \mathds{1}_{2k-1+r}\convolve f_{-X}\right). \]
\end{definition}

\begin{lemma}\label{lem:Bernoulli_calYr_representation}
For every $r \in [0,1]$, the set $\{\hatf_k^r \mid k \in I^r \}$ is a basis for $\aff(\calY^r)$.

That is, for all $r \in [0,1]$ and $f \in \calY^r$, there exist constants $\alpha_k^r = \alpha_k^r[f]$ such that 
\[ f = \sum_{k \in I^r} \alpha_k^r \hatf_k^r. \]
where $\alpha_0^1[f] = \frac{pf(-1)-qf(1)}{p-q}$ and
$ \alpha_k^r[f] = \frac{2(pf(-2k-r) - qf(2k+r))}{p-q}$
for all other $k,r$. 
\end{lemma}

%When appealing to this lemma or \cref{thm:Bernoulli_calY_representation} and 
When there is no danger of confusion we will write $\alpha_k^r$ instead of $\alpha_k^r[f]$. 
The $\alpha_k^0$ are the same as in \cref{lem:calYzero_representation}.

\begin{proof}
The verification is straightforward but tedious. 
The statement for $r=0$ is exactly \cref{lem:calYzero_representation}. Thus, we need only check $r \in (0,1]$. 
First let $r \in (0,1)$ and fix some $z \in \mathbb{R}$. 
If $z \not\in S^r$ then $f(z) = \hatf_k^r(z) = 0$ for all $k \in \mathbb{Z}$. 
Therefore it is sufficient to verify the statement when $z \in S^r$. 
Assume first that $z = 2n+r$ for some $n \in \mathbb{Z}$. 
We have 
\begin{align}
\sum_{k \in \mathbb{Z}} \alpha_k^r \hatf_k^r(2n+r) 
&= \alpha_n^r \hatf_n^r(2n+r) + \alpha_{n+1}^r \hatf_{n+1}^r(2n+r) \nonumber \\
%&= \left(\frac{2}{p-q}\right) \bigg((pf(-2n-r)-qf(2n+r)) \Big(\frac{q}{2}\Big) \nonumber \\
%&\quad +\> (pf(-2n-2-r)-qf(2n+2+r)) \Big(\frac{p}{2}\Big)\bigg) \nonumber \\
&= \left(\frac{2}{p-q}\right) \bigg(\frac{q}{2}(pf(-2n-r)-qf(2n+r)) \nonumber \\
&\qquad \qquad \qquad \quad +\> \frac{p}{2}(pf(-2n-2-r)-qf(2n+2+r))\bigg) \nonumber \\
%&= \left(\frac{1}{p-q}\right) \bigg(p(qf(-2n-r) + pf(-2n-2-r)) \nonumber \\
%&\quad -\> q^2f(2n+r) -pq f(2n+2+r) \bigg) \nonumber \\
&= \left(\frac{1}{p-q}\right) \bigg(p(qf(2n+2+r) + pf(+2n+r)) \nonumber \\
&\qquad \qquad \qquad \quad - q^2f(2n+r) -pq f(2n+2+r) \bigg) \label{eq:calY_repn_intermediate}\\
&= \left(\frac{1}{p-q}\right) (p^2-q^2)f(2n+r) %\nonumber \\
%= f(2n+r) 
= f(z) \nonumber 
\end{align}
where we have used the the symmetry equation 
\[ qf(-2n-r) + pf(-2n-2-r) = pf(2n+r) + qf(2n+2+r) \] 
in \eqref{eq:calY_repn_intermediate}. 
The computation for $z=2n-r$ is almost identical (the nonzero expansion terms are those with $\hatf_{-n}$ and $\hatf_{-n+1}$ instead of $\hatf_n$ and $\hatf_{n+1}$). 
%We include the computation here for completeness. 
%Let $z = 2n-r$ for some $n \in \mathbb{Z}$.
%We have 
%\begin{align*}
%\sum_{k \in \mathbb{Z}} \alpha_k^r \hatf_k^r(2n-r)
%&= \alpha_{-n}^r \hatf_{-n}^r(2n-r) + \alpha_{-n+1}^r \hatf_{-n+1}^r(2n-r) \\
%&= \left(\frac{2}{p-q}\right) \bigg((pf(2n-r)-qf(-2n+r)) \frac{p}{2} \\
%&\quad +\> (pf(2n-2-r)-qf(-2n+2+r)) \frac{q}{2}\bigg) \\
%&= \left(\frac{1}{p-q}\right) \bigg(p^2f(2n-r) - pqf(-2n+r) \\
%&\quad +\> pq(2n-2-r) -q^2 f(-2n+2+r) \bigg) \\
%&= \left(\frac{1}{p-q}\right) \bigg(p^2f(2n-r) + pq(2n-2-r) \\
%&\quad -\> q(pf(-2n+r) + qf(-2n+2+r)) \bigg) \\
%%&= \left(\frac{1}{p-q}\right) \left(p^2f(2n-r) + pq(2n-2-r) - q(f_{X+Y}(-2n+1+r)) \right) \\
%%&= \left(\frac{1}{p-q}\right) \left(p^2f(2n-r) + pq(2n-2-r) - q(f_{X+Y}(2n-1-r)) \right) \\
%&= \left(\frac{1}{p-q}\right) \bigg(p^2f(2n-r) + pq(2n-2-r) \\
%&\quad -\> q(qf(2n-r) + pf(-2n-2-r)) \bigg) \\
%&= \left(\frac{1}{p-q}\right) (p^2-q^2)f(2n-r) \\
%&= f(2n-r) = f(z) 
%\end{align*}
%Thus for any $r\in(0,1)$ and $z \in \mathbb{R}$, $f(z) = \sum_{n \in \mathbb{Z}} \alpha_n^0 \hatf^r_n(z)$.

It remains only to show the expansion $f = \sum_{k \ge 0} \alpha_k^r \hatf_k^r$ for $r=1$. 
Again we need only consider $z \in S^r = S^1 = 2\mathbb{Z}+1$; therefore we write $z=2n+1$. 
The computation is the same as above when $n \ge 1$ due to the definition of the coefficients $\alpha_k^r$. 
For $n \le -2$, we compute
\begin{align*}
\sum_{k \ge 0} \alpha_k^1 \hatf_k^1(z) 
&= \sum_{k \ge 0} \alpha_k^1 \hatf_k^1(2n+1) \\
&= \alpha_{-n-1}^1 \hatf_{-n-1}^1(2n+1) + \alpha_{-n}^1 \hatf_{-n}^1(2n+1) \\
&= \left(\frac{2}{p-q}\right)\bigg( \frac{p}{2}(pf(2n+1)-qf(-2n-1)) %\\
+ \frac{q}{2}(pf(2n-1)-qf(-2n+1))\bigg) \\
%&= \left(\frac{1}{p-q}\right)\left( (pf(2n+1)-qf(-2n-1))p + (pf(2n-1)-qf(-2n+1))q\right) \\
&= \left(\frac{1}{p-q}\right)\left( p^2f(2n+1) + pqf(2n-1) -q(pf(2n-1) +qf(2n+1)) \right) \\
&= \left(\frac{1}{p-q}\right)\left( (p^2-q^2)f(2n+1) \right) %\\
%&= f(2n+1) 
= f(z) . 
\end{align*}
For $n=0$ (i.e. $z=1$) we compute
\begin{align*}
\sum_{k \ge 0} \alpha_k^1 \hatf_k^1(z)
&= \alpha_0^1 \hatf_0^1(1) + \alpha_1^1 \hatf_1^1(1) \\
&= \left(\frac{1}{p-q}\right)(pf(-1)-qf(1))q + \left(\frac{2}{p-q}\right)(pf(-3)-qf(3))\frac{p}{2} \\
&= \left(\frac{1}{p-q}\right)(pqf(-1)-q^2f(1) + p^2f(-3) -pqf(3)) \\
&= \left(\frac{1}{p-q}\right)(p^2f(1)-q^2f(1) + pqf(3) -pqf(3)) %\\
%= f(1) 
= f(z). 
\end{align*}
Similarly, when $n=-1$ (i.e. $z=-1$) we compute
\begin{align*}
\sum_{k \ge 0} \alpha_k^1 \hatf_k^1(z)
&= \alpha_0^1 \hatf_0^1(-1) + \alpha_1^1 \hatf_1^1(-1) \\
&= \left(\frac{1}{p-q}\right)(pf(-1)-qf(1))p + \left(\frac{2}{p-q}\right)(pf(-3)-qf(3))\frac{q}{2} \\
%&= \left(\frac{1}{p-q}\right)(p^2f(-1)-pqf(1) + pqf(-3) -q^2f(3)) \\
&= \left(\frac{1}{p-q}\right)(p^2f(-1)-q^2f(-1) + pqf(-3) -pqf(-3)) %\\
%&= f(-1) 
= f(z). \qedhere
\end{align*}
%Thus we have established $f(z) = \sum_{k\ge0} \alpha_k^1 \hatf_k^1(z)$ for all $z\in\mathbb{R}$. 
\end{proof}

\begin{theorem}\label[theorem]{thm:Bernoulli_calY_representation}
The set $\{\hatf_k^r \mid r \in [0,1] \text{ and } k \in I^r \}$ is a basis for $\aff(\calY(f))$. 

That is, for $f \in \calY$, denoting 
%the set of all indices $r$ such that the projection of $f$ onto $S^r$ is nonzero by 
$ R_f = \{ r \in [0,1] \mid \supp(f) \cap S^r \ne \emptyset \}, $ we have
\[ f = \sum_{r \in R_f} \sum_{k \in I^r} \alpha_k^r \hatf_k^r , \qquad \text{and} \qquad \sum_{r \in R_f} \sum_{k \in I^r} \alpha_k^r = 1 . \]
\begin{proof}
The PMF $f$ is discrete, so its support is at most countable; therefore $R_f$ is at most countable. 
Write $f = \sum_{R_f} f\big\vert_{S^r}$ and for each $r \in R_f$ define $c_r = \sum_{z \in S^r} f(z)$. 
Then for each $r \in R_f$, $f_r = \frac{1}{c_r} f\big\vert_{S^r} \in \calY$ by \cref{lem:Sr_restriction_of_symmetrizers_produces_symmetrizers}. 
Now apply the \cref{lem:Bernoulli_calYr_representation} to each $f_r$ to find constants $\alpha_k^r[f_r]$; it is then obvious that $c_r \alpha_k^r[f_r] = \alpha_k^r[f]$ for all $k$ and $r$. 
%; then, for $(k,r) \ne (0,1)$,
%\begin{align*}
%\[ c_r \alpha_k^r[f_r]
%= c_r \frac{2}{p-q}(pf_r(-2k-r) - qf_r(2k+r)) \\
%= \frac{2}{p-q}(pc_rf_r(-2k-r) - qc_rf_r(2k+r))
%= \frac{2}{p-q}(pf(-2k-r) - qf(2k+r))
%= \alpha_k^r[f] \]
%\end{align*}
%and similarly $c_r\alpha_1^0[f_r] = \alpha_1^0[f]$. 
Thus we compute 
\[ 
f = \sum_{r \in R_f} c_r f_r %&= \sum_{r \in R_f} c_r \sum_{k \in I^r} \alpha_k^r[f_r] \hatf_k^r \\
= \sum_{r \in R_f} \sum_{k \in I^r} c_r \alpha_k^r[f_r] \hatf_k^r 
= \sum_{r \in R_f} \sum_{k \in I^r} \alpha_k^r \hatf_k^r 
\]
as needed. 
Finally, 
\begin{align*} 
\sum_{r \in R_f} \sum_{k \in I^r} \alpha_k^r 
&= \sum_{r \in R_f} \sum_{k \in I^r} \alpha_k^r \bigg( \sum_{z \in \supp(f)} \hatf_k^r(z) \bigg) \\
&= \sum_{z \in \supp(f)} \bigg(\sum_{r \in R_f} \sum_{k \in I^r} \alpha_k^r \hatf_k^r(z)\bigg) 
= \sum_{z \in \supp(f)} f(z)
= 1 . \qedhere
\end{align*}
\end{proof}
\end{theorem}

The next result is an extension of \cref{lem:calYzero_negative_coefficient_control}. 
The proof is the same. 
\begin{lemma}\label[lemma]{lem:calY_negative_coefficient_control}
Let $f \in \calY$ and use \cref{thm:Bernoulli_calY_representation} to write 
$$ f = \sum_{r \in R_f} \sum_{k \in I^r} \alpha_k^r \hatf_k^r . $$ 
Then, $\alpha_1^0 \ge 0$, and for any $r,j$ such that $\alpha_j^r \le 0$, we have 
$\alpha_{j+1}^r \ge \frac{p}{q}\abs{\alpha_j^r}$, and also
$\alpha_{j-1}^r \ge \frac{p}{q}\abs{\alpha_j^r}$ when $j-1 \in I^r$ (i.e., when $(r,j) \not\in \{(0,1),(1,0)\}$). 
\end{lemma}

%For applications of \cref{thm:Bernoulli_calY_representation} to the extreme points of $\calY$, see \cref{sec:Bernoulli_extreme_points}. or... the thesis of Pollard. 

\section{Variance symmetrization resistance in the reals}
\label{sec:var-reals}

\subsection{Variance symmetrization resistance of Bernoulli: an alternative approach}
\label[subsection]{sec:Bernoulli_variance}

We now continue with a discussion of variance symmetrization resistance. 
Denote the $k$th moment of a PMF $f$ by $M_k(f)$. 
%$$ M_k(f) = \sum_{z \in \supp(f)} z^k f(z). $$
We record the first and second moments of the functions $\hatf_k^r$ of \cref{def:Bern_hatfkr}.
\begin{lemma}\label[lemma]{lem:hatf_moments}
For $r \in [0,1]$ and $k \in I^r$, 
$$ M_1(\hatf_k^r) = -M_1(X) = q-p $$
and 
$$ M_2(\hatf_k^r) = 4k^2 -4k +2 + (4k-2)r + r^2. $$
%In particular, $M_2(f_0^1) = M_2(f_{-X}) = 1$. 
%\begin{proof}
%The result for the first moment follows directly from the fact that $\hatf_k^r \in \calY$ and $\mathbb{E}(X+Y) = 0$ for any $Y \sim f$ independent of $X$. 
%
%The result for the second moment is a straightforward computation from the definition of the $\hatf_k^r$. 
%We note that the described formula for $M_2(\hatf_k^r)$ also applies to the edge cases $M_2(f_0^1) = M_2(f_{-X}) = 1$ and $M_2(\hatf_1^0) = 2$. 
%\end{proof}
\end{lemma}

\begin{corollary}
For $k$ and $r$ such that $\hatf_k^r$ is defined, 
$$ \Var(\hatf_k^r) = 4k^2 -4k +2 + (4k-2)r + r^2 - (p-q)^2. $$
\end{corollary}

\begin{lemma}\label[lemma]{lem:calYr_min_M2s}
Let $f \in \calY^r$. 
\begin{enumerate}[(i)]
\item[(i)] If $r=0$, then $M_2(f) \ge 2$, with equality if and only if $f = \hatf_1^0$. 
\item[(ii)] If $r\in(0,1)$, then $M_2(f) \ge 2-2r+r^2 > 1$, with equality if and only if $f = \hatf_0^r$. 
\item[(iii)] If $r=1$, then $M_2(f) \ge 1$, with equality if and only if $f = \hatf_0^1$. 
\end{enumerate}

\begin{proof}
First consider $r=0$. We compute
\begin{align}
M_2(f) 
%&= \sum_{z \in 2\mathbb{Z}} z^2 f(z) \\
= \sum_{\substack{z \in 2\mathbb{Z} \\ z \ne 0}} z^2 f(z)
&\ge \sum_{\substack{z \in 2\mathbb{Z} \\ z \ne 0}} 2^2 f(z) \label{eq:coord_redn_r_ineq1} \\
&= 4(1-f(0)) \nonumber \\
%&= 4\bigg(1-\frac{\alpha_1^0}{2}\bigg) \label{eq:coord_redn_r_eq1} \\
&= 4\bigg(1-\frac{1-\sum_{k \ge 2}\alpha_k^0}{2}\bigg) \label{eq:coord_redn_r_eq2}\\
&\ge 4\bigg(1-\frac{1}{2}\bigg) 
= 2 
= M_2(\hatf_1^0) .\label{eq:coord_redn_r_ineq2}  
\end{align}
where we have used 
\cref{thm:Bernoulli_calY_representation} in %\eqref{eq:coord_redn_r_eq1} and 
\eqref{eq:coord_redn_r_eq2},
and \cref{lem:calY_negative_coefficient_control} in \eqref{eq:coord_redn_r_ineq2}.
Equality is attained in \eqref{eq:coord_redn_r_ineq1} only if $\supp(f) \subseteq \{0, \pm 2\}$; %in turn 
which forces $f = \hatf_1^0$ by 
\cref{thm:Bernoulli_calY_representation}.
%shows that this implies $f = \hatf_1^0$. 
%by \cref{thm:Bernoulli_calY_representation}.
%, which attains equality in \eqref{eq:coord_redn_r_ineq2} as well. 
%Thus $M_2(f) = M_2(\hatf_1^0)$ if and only if $f = \hatf_1^0$. 

Now consider $r \in (0,1)$. 
%First note that $M_2(\hatf_0^r) = 2-2r+r^2$ from \cref{lem:hatf_moments}. 
Noting that the four smallest elements in magnitude of $S^r = 2\mathbb{Z} + \{\pm r\}$ are $\{\pm r, \pm(2-r)\}$, we compute
\begin{align*}
M_2(f) 
&= \sum_{z \in S^r} z^2 f(z) \\
%&= \sum_{z \in 2\mathbb{Z} + \{\pm r\}} z^2 f(z) \\
%&= \sum_{m \in 2\mathbb{Z}} (2m+r)^2 f(2m+r)
%    + \sum_{m \in 2\mathbb{Z}} (2m-r)^2 f(2m-r) \\
%&= r^2f(r) + \sum_{\substack{m \in 2\mathbb{Z}\\ m \ne 0}} (2m+r)^2 f(2m+r)
%    + (-r)^2f(-r) + \sum_{\substack{m \in 2\mathbb{Z}\\ m \ne 0}} (2m-r)^2 f(2m-r) \\
&\ge %r^2(f(r)+f(-r)) + \sum_{\substack{m \in 2\mathbb{Z}\\ m \ne 0}} (-2+r)^2 f(2m+r)
%    + \sum_{\substack{m \in 2\mathbb{Z}\\ m \ne 0}} (2-r)^2 f(2m-r) \\
%&= 
r^2(f(r)+f(-r)) + \sum_{\substack{z \in S^r \\ z \ne \pm r}} (2-r)^2 f(z) \\
&= r^2(f(r)+f(-r)) + (2-r)^2(1-f(r)-f(-r)) \\
%&= (r^2 - 4 +4r -r^2)(f(r) + f(-r)) + 4-4r+r^2 \\
%&= (-4 +4r)(f(r) + f(-r)) + 4-4r+r^2 \\
%&= 4(-1)(1-r)(f(r) + f(-r)) + 4(1-r)+r^2 \\
%&= 4(1-r)(1 + (-1)(f(r) + f(-r))) + r^2 \\
&= 4(1-r)(1 - f(r) - f(-r)) + r^2 .
\end{align*}
By \cref{lem:calY_negative_coefficient_control}, 
%$\sum_{\substack{k \in \mathbb{Z} \\ k \ne 0,1}} \alpha_k^r \ge 0$, 
%so that 
\( 
\alpha_0^r +\alpha_1^r 
= 1-\sum_{\substack{k \in \mathbb{Z} \\ k \ne 0,1}} \alpha_k^r
\le 1,
\)
and 
$$f(r) 
= \alpha_0^r \hatf_0^r(r) + \alpha_1^r \hatf_1^r(r) 
= \alpha_0^r (\frac{q}{2}) + \alpha_1^r (\frac{p}{2})$$
and similarly 
$$f(-r) 
= \alpha_0^r \hatf_0^r(-r) + \alpha_1^r \hatf_1^r(-r) 
= \alpha_0^r (\frac{p}{2}) + \alpha_1^r (\frac{q}{2})$$
so that 
$f(r) + f(-r) 
%= \alpha_0^r (\frac{q}{2} + \frac{p}{2}) + \alpha_1^r (\frac{p}{2} + \frac{q}{2}) 
%= \alpha_0^r (\frac{p+q}{2}) + \alpha_1^r (\frac{p+q}{2}) 
= \frac{1}{2}(\alpha_0^r + \alpha_1^r)
\le \frac{1}{2}
$. 
Therefore
\begin{align*}
M_2(f) 
&\ge 4(1-r)(1-\frac{1}{2}) + r^2 
= 2(1-r) + r^2 
%&= 2 - 2r + r^2 
= M_2(\hatf_0^r) . 
\end{align*}
Equality is attained only if $\supp(f) \subseteq \{\pm r, \pm(2-r)\}$, which forces $f = \hatf_0^r$ by \cref{thm:Bernoulli_calY_representation}. 
%The converse is obvious. 

Finally consider $r=1$. 
In this case $\supp(f) = 2\mathbb{Z}+\{1\}$, so 
\[ 
M_2(f) = \sum_{z \in \supp(f)} z^2 f(z) 
\ge \sum_{z \in \supp(f)} f(z) = 1 = M_2(\hatf_0^1). 
\]
In order to attain equality we need $z^2 = 1$ for every $z \in \supp(f)$; in this case $\supp(f) \subseteq \{\pm 1\}$ and $f = \hatf_0^1$ by \cref{lem:support_size_one_or_two}. 
\end{proof}
\end{lemma}

We now give a new proof of the main theorem of \cite{KMSVV99,Pal08}. 
\begin{theorem}[Variance symmetrization resistance of Bernoulli]\label{thm:VarY_exceeds_VarX}
Let $X \sim \Bern(p,a,b)$ for some $p \ne \frac{1}{2}$. 
Then, any independent symmetrizer $Y$ of $X$ satisfies $\Var(Y) \ge \Var(X)$, with equality if and only if $f_Y = f_{-X}$.

\begin{proof}
By \cref{lem:affine_invariance} we may take $X \sim \Bern(p,-1,1)$. 
%We show that every $f \in \calY$ satisfies $\Var(f) \ge \Var(f_{-X}) = 4pq$, with equality if and only if $f = f_{-X}$. 
Let $f \in \calY$, 
define $R_f$ as in \cref{thm:Bernoulli_calY_representation}, 
and define %$ f_r(\cdot) = f(\cdot)\mathds{1}_{I^r}(\cdot)$ 
for each $r \in R_f$ constants $c_r \coloneqq \sum_{z \in S^r}f(z)$ and PMFs $g_r(\cdot) \coloneqq \frac{1}{c_r}f(\cdot)\mathds{1}_{S^r}(\cdot) \in \calY^r$ as in \cref{lem:Sr_restriction_of_symmetrizers_produces_symmetrizers}. 
%Then $f = \sum_{r \in R_f} f_r$. 
%Define 
%Then for each $r \in R_f$, $g_r \in \calY^r$ by \cref{lem:Sr_restriction_of_symmetrizers_produces_symmetrizers}. 
Then $f = \sum_{r \in R_f} c_r g_r$ and, by the concavity of the variance,
\begin{align}
\Var(f) 
\ge \sum_{r \in R_f} c_r \Var(g_r) \nonumber
&= \sum_{r \in R_f} c_r (M_2(g_r) - M_1(g_r)^2) \nonumber \\
\label{eq:varsymmres2} &\ge \sum_{r \in R_f} c_r (1 - (p-q)^2) \\
%\ge \sum_{r \in R_f} c_r - (p-q)^2
%= \sum_{r \in R_f} c_r M_2(g_r) - (p-q)^2 \sum_{r \in R_f} c_r
%\ge \sum_{r \in R_f} c_r - (p-q)^2
&= 1 - (p-q)^2 
= \Var(f_{-X}) \nonumber
\end{align}
where we have used \cref{lem:calYr_min_M2s} in (\ref{eq:varsymmres2}). %the second inequality. 
%Finally note that $1-(p-q)^2 = (p+q)^2 - (p-q)^2 = 4pq$. 

Note that by \cref{lem:calYr_min_M2s}, equality is attained 
%in (\ref{eq:varsymmres1}) if and only if for all $r \in R_f$, $\Var(f) = \Var(g_r)$. 
in (\ref{eq:varsymmres2}) if and only if for all $r \in R_f$, $M_2(g_r) = 1$. 
%According to \cref{lem:calYr_min_M2s}, this only happens when $r=1$ and 
%This forces $g_r = f_1^k$ for every $r \in R_f$;
In this case $R_f = \{1\}$ and $f = g_1 = \hatf_0^1$. %; the converse equality condition is obvious. 
%To see when $\Var(f) = \Var(f_{-X})$, 
%note that since the mean of every function in $\calY$ is fixed at $-\mathbb{E}X$,  \cref{lem:calYr_min_M2s} enumerates the minimum-variance functions in each $\calY^r$. 
%Thus the minimum-variance function in each $\calY^r$ has variance $\varphi(r) = 2-2r+r^2 - (\mathbb{E}X)^2$. % for all $r \in [0,1]$ according to \cref{lem:calYr_min_M2s}.
%Clearly $\varphi(r)$ attains its unique minimum on $[0,1]$ at $r=1$, so the unique variance-minimizing element of the spaces $\calY^r$ is $\hatf_k^1$. 
%As $f$ is a convex combination of functions from the spaces $\calY^r$, the variance of $f$ is minimized if and only if $f = \hatf_0^1$. 
%Since $f = \sum c_r g_r$ with each $g_r \in \calY^r$, $\Var(f)$ is minimized if and only if $f = \hatf_0^1$.  
%Thus \[ \Var(X) \ge c_1\Var(g_1) + (1-c_1)\Var(h_1) \]
%the minimizer $f$ of $\Var(f)$ must lie in some $\calY^r$ by virtue of the concavity of variance and the convex combination $f = \sum_{r \in R_f} c^r g^r$. 
%minimum-variance $2-2r+r^2$ is decreasing on $(0,1)$ and does not attain its minimum value of $2-2(1)+(1)^2 = 1$ on that interval. 
%By \cref{lem:calYr_min_M2s}, then, equality in $\Var(f) \ge \Var(f_{-X})$ implies $R_f = \{1\}$. 
%By the equality condition for $\calY^1$ in \cref{lem:calYr_min_M2s}, we must then also have $f = \hatf_0^1 = f_{-X}$. 
%The converse is obvious. 
\end{proof}
\end{theorem}

\subsection{Nonnegative random variables on the real line} 
In the present section, we explore how far the method of Pal \cite{Pal08} can be adapted to address variance symmetrization resistance for nonnegative random variables.

\begin{theorem}\label{thm:V-symm-nonneg}
Let $X \sim \mu$ be supported on $[0,\infty)$ with $\mathbb{E}X =m_1$. Then $X$ is $c$-Variance symmetrization resistant for 
\[ c = 1- \frac{2}{\Var(X)} \int_0^{m_1} (x-m_1)^2 d\mu(x) . \]
\end{theorem}
\begin{proof}
Let $X \sim \mu$ be supported on $[0,\infty)$ with $\mathbb{E}X =m_1$ and $\Var(X) =m_2$. 
Let $Y \sim f_Y \in \calY(X)$ with finite variance. Note $\mathbb{E}Y = -\mathbb{E}X = -m_1$. 
By Skorokhod embedding, there exists a stopping time $T$ of the standard Wiener process $W_t$ such that $W_T \overset{D}{=} Y-\mathbb{E}Y$. 
Thus for the Wiener process $B_t = W_t + (X-\mathbb{E}X)$, we have $B_T \overset{D}{=} X+Y$. 
Note also that $B_0\overset{D}{=} X -m_1$, and write $\tilde{\mu}$ for the distribution of $B_0$ (which is just a shifted version of $\mu$ supported on some subset of $[-m_1, \infty)$). 

Let $\rho \colon \mathbb{R} \rightarrow \mathbb{R}$ be odd and $C^2$ with $\abs{\rho''} \le 1$. 
By It\^o's rule, there exists a martingale $M_t$ such that 
\[ \rho(B_t) - \rho(B_0) = M_t + \frac{1}{2} \int_0^t \rho''(B_s) ds \]
taking expectations, 
\[ \mathbb{E}\rho(B_T) - \mathbb{E}\rho(B_0) = \mathbb{E}M_T + \frac{1}{2} \mathbb{E} \int_0^T \rho''(B_s) ds.  \]
Since $X+Y$ is symmetric and $\rho$ is odd, $\mathbb{E}\rho(B_T) = 0$, and by optional sampling, $\mathbb{E}M_T=0$, so that
\begin{align*}
%2\int_{-m_1}^{\infty}\rho(x) d\tilde{\mu}(x)
%- 2 \mathbb{E}\rho(X-\mathbb{E}X) 
2 \mathbb{E}\rho(B_0) 
&= -\mathbb{E} \int_0^T \rho''(B_s) ds \\
&= \int_{\mathbb{R}}  \mathbb{E}\left[ \int_0^T -\rho''(B_s) ds \mid B_0 = x \right] d\tilde{\mu}(x) \\
%&= \int_{-1}^{\infty} e^{-(x+1)} \mathbb{E}\left[ \int_0^T \rho''(x+W_s) ds \mid B_0 = x \right] dx \\
&\le \int_{\mathbb{R}}  \mathbb{E}\left[ \int_0^T \abs{-\rho''(B_s)} ds \mid B_0 = x \right] d\tilde{\mu}(x) \\
&\le \int_{\mathbb{R}} \mathbb{E}\left[ \int_0^T 1 ds \mid B_0 = x \right] d\tilde{\mu}(x) \\
&= \int_{\mathbb{R}} \mathbb{E}[T\mid B_0 = x] \cdot d\tilde{\mu}(x)\\
&= \mathbb{E} \big[ \mathbb{E}[T|B_0] \big] \\
&=  \mathbb{E}T
%&= \Var(Y) \int_{-1}^{\infty} e^{-(x+1)} dx \\
= \Var(Y) .
\end{align*}

Let us now consider a specific choice of $\rho$:
$$
\rho(x)=\frac{1}{2}\big[ (x-m_1)^2-m_1^2\big] 
$$
for $x\geq 0$, and defined via the oddness property $\rho(x) \coloneqq -\rho(-x)$ for $x < 0$. 
Clearly, this $\rho$ is continuous and odd; moreover,  $\rho''(x)\in [-1,1]$ for all $x\neq 0$. It is also easy to check that $\rho$ is continuously differentiable at 0, but $\rho''(0)$ does not exist and hence $\rho$ is not $C^2$. Nonetheless, $\rho''$ does exist in the sense of Schwarz distributions. Therefore, applying Aebi's \cite{Aeb92} extension of It\^o's formula, the above argument still holds.

We can write, for all $x\in \R$,
$$
2\rho(x)= \big[ (x-m_1)^2-m_1^2\big] - \tilde{\rho}(x) ,
$$
where $\tilde{\rho}(x)=2x^2$ if $x<0$,
and $\tilde{\rho}(x)=0$ otherwise.
Thus we have
\begin{align*}
\Var(Y) 
&\geq  \mathbb{E} [2\rho(X-m_1)] \\
&= \mathbb{E}\big[ (X-2m_1)^2-m_1^2\big] - \mathbb{E}\tilde{\rho}(X-m_1)\\
&=\mathbb{E}(X^2)+4m_1^2 -4m_1 \mathbb{E}(X)-m_1^2- \int_0^{m_1} 2(x-m_1)^2 d\mu(x) \\
%&= m_2 -m_1^2 +\int \\
&= \Var(X) - 2\int_0^{m_1} (x-m_1)^2 d\mu(x) 
\end{align*}
so that $X$ is $c$-Variance symmetrization resistant for 
\[ c = 1- \frac{2}{\Var(X)} \int_0^{m_1} (x-m_1)^2 d\mu(x) . \]
\end{proof}

For example, for $X \sim \Exp(1)$,
\[ c = 1- \frac{2}{1} \int_0^1 (x-1)^2 e^{-x} dx \approx 0.4715, \]
for $X \sim \gamma(2,1)$,
\[ c 
= 1-\frac{2}{2} \int_0^2 (x-2)^2 \frac{1}{\Gamma(2)} x e^{-x} dx \approx 0.3534, 
\]
and for $X \sim \Poisson(1)$,
\[ c
= 1-\frac{2}{e} \sum_{k=0}^1 \frac{(k-1)^2}{k!} 
= 1-\frac{2}{e} 
\approx 0.2642 .
\]

It is easy to see that Pal's method cannot possibly achieve a constant $c=1$ above for any choice of function $\rho$. 
Nonetheless, we conjectured that the exponential distribution is Variance-symmetrization-resistant. 
We communicated this conjecture to Jiange Li, who resolved it in the affirmative for the case of absolutely continuous symmetrizers \cite{Li26} using different ideas.

\section[Entropic symmetrization resistance on the hypercube]{Entropic symmetrization resistance on $\{0,1\}^d$}
\label{sec:Bernoulli_entropy}
%\label[subsection]{sec:Bernoulli_entropy}

\subsection{Dimension 1}

We wish to investigate symmetrization resistance of asymmetric $X \sim \Bern(p,a,b)$. 
By \cref{lem:affine_invariance} we may assume that $X\sim\Bern(p,-1,1)$; \cref{cor:psymmetry} makes it clear that we may also assume $p > \frac{1}{2}$. 
We will write $q = 1-p$ for ease of notation. 
In particular, $p = \max(p,q)$ and $q = \min(p,q)$. 

%We begin with an elementary lemma. 

\begin{lemma}\label[lemma]{lem:support_size_one_or_two}
All $f \in \calY$ satisfy $|\supp(f)| \ge 2$, 
and if $|\supp(f)| = 2$ then $f = f_{-X}$.  

\begin{proof}
Taking $Y \sim f$, the first assertion is clear: $Y$ a constant would force $Y = 0$ (as the support of $X+Y$ must be symmetric), 
%and then $p = \frac{1}{2}$, 
contradicting asymmetry of $X$. 

Now consider $\supp(f) = \{y_1, y_2\}$ with $y_1 < y_2$.
By symmetry of $X+Y$, $(1-p)f(y_1) = pf(y_2)%$, so that $(1-p)f(y_1) 
= p(1-f(y_1))$; hence $f(y_1)=p$. 
%Thus $f(y_2) = q$. 
\end{proof}
\end{lemma}

\begin{lemma}\label[lemma]{no_atoms_bigger_than_p}
Let $f \in \calY$ with $\abs{\supp(f)} \ge 3$. 
For all $y \ne 0$, $f(y) < p$. 
\begin{proof}
We separate the proof into two cases, when $y=1$ and when $y \ne 1$. 
In the case where $y \ne 1$, 
\begin{align*}
q f(-y+2) + p f(-y)
&=  f_{X+Y}(-y+1) \\
&= f_{X+Y}(y-1) = q f(y) + p f(y-2) 
\ge qf(y). 
\end{align*}
Dividing both sides by $p$ and adding $f(y)$, 
\[
f(y) + \frac{q}{p} f(-y+2) + f(-y) \ge \frac{q}{p}f(y) + f(y) = \frac{1}{p}f(y) .
\]
Since $\frac{q}{p} < 1$ and the values $y$, $-y+2$, and $-y$ are distinct, the left-hand side of the preceding equation is strictly bounded above by $1$. 
%\(
%1 \ge f(y) + f(-y+2) + f(-y) \ge \frac{1}{p}f(y) .
%\)

Similarly, when $y=1$, we have 
\[
qf(-1) + pf(-3) 
= f_{X+Y}(-2) 
= f_{X+Y}(2)
= qf(3) + pf(1) 
\ge pf(1) 
\]
so that
\[
f(-1) + f(-3) \ge \frac{q}{p} f(-1) + f(-3) \ge f(1) . 
\]
Adding $f(1)$ to both sides, we see that 
\(
1 \ge f(1) + f(-1) + f(-3) \ge 2f(1)
\);
in particular $f(1) \le \frac{1}{2} < p$. 
\end{proof}
\end{lemma}

\begin{lemma}\label[lemma]{unique_large_mass}
Let $f \in \calY$ with $\abs{\supp(f)} \ge 3$. 
At least one of of the following holds: 
\begin{enumerate}[(i)]
\item $H(f) > H(X)$
\item$f(0) \in (0,q) \cup (p,1)$, and for all $y \ne 0$, $f(y) \in (0,q)$.
\end{enumerate}

\begin{proof}
Fix $y \in \mathbb{R}$. 
Define the \emph{Bernoulli entropy function} 
$H_B(t) = H(\Bern(t))$ for $t \in (0,1)$. % and define $H_B(0) = H_B(1) = 0$ so that $H_B$ is continuous on $[0,1]$. 
It is easy to see that $H_B$ is concave, symmetric about $t=\frac{1}{2}$, and achieves its maximum at $t=\frac{1}{2}$. 
%Because $H(Y) \ge H(\varphi(Y))$ for any function $\varphi$, we know that 
If there exists $y$ such that $f(y) \in [q, p]$, then, writing $Y \sim f$, we have $H(Y) > H(\mathds{1}\{Y=y\}) = H_B(f(y)) \geq H_B(p) = H(X)$. 
The result now follows from \cref{no_atoms_bigger_than_p}. 
\end{proof}
\end{lemma}

The next lemma shows that any counterexample to the entropic symmetrization resistance of Bernoullis must be concentrated at its median (which must be unique). 

\begin{lemma}\label[lemma]{large_median_at_zero}
Let $f \in \calY$ with $\abs{\supp(f)} \ge 3$. 
Then, $H(f) > H(X)$ or $f(0) > p$. 
\begin{proof}
If $f < \frac{1}{2}$ everywhere,
\[
H(f) = \sum_{y} f(y) \log_2(1/f(y)) 
\ge \sum_y f(y) = 1 > H(X). 
\]
Otherwise, we apply \cref{unique_large_mass}. 
\end{proof}
\end{lemma}

Now we show that elements of $\calY$ have (rescaled) projections in the spaces $\calY^r$ where they are supported. 
\begin{lemma}\label[lemma]{lem:Sr_restriction_of_symmetrizers_produces_symmetrizers}
Let $f \in \calY$ and let $r > 0$ satisfy $\supp(f) \cap S^r \ne \emptyset$. 
Then 
\( g := (\frac{1}{c}) f \mathds{1}_{S^r} \in \calY^r, \)
where $c = \sum_{z \in S^r} f(z)$. 
Moreover, if $c < 1$, then $\frac{f-cg}{1-c} \in \calY$.

\begin{proof}
Fix $r \in [0,\infty)$ and denote $ F_z = \{\pm z,-z\pm2,z\pm2\}$ for $z \in \mathbb{R}$. 
Writing one of $z = 2k - r$ (or $z = 2k + r$) for some $k \in \mathbb{Z}$, we have $-z = -2k+r$ (or $-z = -2k -r$); either way, $-z \in S^r$. 
Also, $z \in S^r$ implies $\{z \pm 2\} \subseteq S^r$ and $-z \in S^r$ implies $\{-z \pm 2\} \subseteq S^r$. 
Thus $z \in S^r$ if and only if $F_z \subseteq S^r$. 

Next, write $g = (\frac{1}{c}) f \mathds{1}_{S^r}$ and consider the symmetry equations 
\begin{align}
\label{symmeqa} q\varphi(z+2) + p\varphi(z) 	&= q\varphi(-z) + p\varphi(-z-2) \\
\label{symmeqb} q\varphi(z) + p\varphi(z-2)   	&= q\varphi(-z+2) + p\varphi(-z)  . 
\end{align}
Note that $\varphi$ is only evaluated at elements of $F_z$ in (\ref{symmeqa}-\ref{symmeqb}). 
When $z \in S^r$, clearly $\varphi=f$ solves (\ref{symmeqa}-\ref{symmeqb}); since both equations are linear and $g = \frac{1}{c}f$ on $F_z$, $g$ is also a solution. 
When $z \not\in S^r$, $g = 0$ on $F_z$ and solves (\ref{symmeqa}-\ref{symmeqb}) trivially. 
Thus $g \in \calY$, and indeed $g \in \calY^r$.

Finally, when %$\supp(f) \setminus S^r \ne \emptyset$, we have 
$c < 1$, we have $(\frac{1}{1-c}) f \mathds{1}_{\mathbb{R} \setminus S^r} = \frac{f-cg}{1-c} \in \calY$ by \cref{lem:closure_under_subtraction}.
\end{proof}
\end{lemma}

%The following lemma shows that %to establish entropic symmetrization resistance 
The next result shows that it is sufficient to investigate $\calY^0$. 
\begin{lemma}\label[lemma]{calYzero_sufficient}
Let $f \in \calY$ with $\abs{\supp(f)} \ge 3$. 
At least one of of the following holds: 
\begin{enumerate}[(i)]
\item $H(f) > H(X)$
\item $g = \big(\sum_{k \in S^0} f(k)\big)^{-1} f \mathds{1}_{S^0} \in \calY^0$ and $H(g) \le H(X)$.
\end{enumerate}

\begin{proof}
Note that $S^0 = 2 \mathbb{Z}$. 
%The case $f = f_{-X}$ is trivial; so we assume hereafter than $f \ne f_{-X}$. 
By \cref{large_median_at_zero}, $c = \sum_{i \in 2\mathbb{Z}} f(i) > 0$, 
so $g \in \calY^0$ by \cref{lem:Sr_restriction_of_symmetrizers_produces_symmetrizers}. 
%We only need to show that $H(g) \le H(X)$. 
%Indeed, 
If $c = 1$, $H(g) = H(f) \le H(X)$. 
Otherwise, $c \in (0,1)$ and 
$ f = cg + (1-c)h $,
where 
$ h = \frac{f-cg}{1-c}$ is a PMF according to \cref{lem:Sr_restriction_of_symmetrizers_produces_symmetrizers}. 
%since clearly $1-c = \sum_{y \in \supp(f) \setminus 2\mathbb{Z}} f(y)$, $h$ is a PMF. 
%We don't actually use the fact that h \in \calY anywhere, so we don't need to prove it here, though it is true. 
Noting that $h(0) = 0$, \cref{large_median_at_zero} shows that $H(h) > H(X)$; by concavity of entropy, therefore, 
\[ H(f) > cH(g) + (1-c)H(h) > cH(g) + (1-c)H(X) .\] 
Thus either $H(f) > H(X)$ or $H(X) \ge H(f) > cH(g) + (1-c)H(X)$. 
%By concavity of entropy
%\begin{align*} 
%H(X) \ge H(f) &> cH(g) + (1-c)H(h) .
%\intertext{Because $h(0) = 0$ by construction, \cref{large_median_at_zero} shows that 
%either $h = f_{-X}$ and $H(h) = H(X)$, or else $h \ne f_{-X}$ and $H(h) > H(X)$. 
%In either case, 
%$H(h) \ge H(X)$, so }
%H(X) &> cH(g) + (1-c)H(X) . \qedhere
%\end{align*}
%so $H(X) > H(g)$.
\end{proof}
\end{lemma}

Although $\calY^0$ is a convex space, the coefficients $\alpha_i^0$ in \cref{lem:calYzero_representation} may be negative.
For example, for any positive integer $n$ we have $(2-\frac{q}{p})^{-1}(\hatf_n^0 - \frac{q}{p}\hatf_{n+1}^0 + \hatf_{n+2}^0) \in \calY^0$. 
However, as we now demonstrate, %there cannot be too many negative $\alpha_i^0$, and 
the magnitude of any negative $\alpha_i^0$ is smaller than that of its (necessarily positive) neighbors $\alpha_{i\pm1}^0$. %, which are always positive. 

\begin{lemma}\label[lemma]{lem:calYzero_negative_coefficient_control}
Let $f \in \calY^0$ and write
\( f = \sum_{k=1}^{\infty} \alpha_k^0 \hatf_k^0 \)
as in \cref{lem:calYzero_representation}.
Then $\alpha_1^0 \ge 0$, and for all integers $j \ge 2$ such that $\alpha_j^0 \le 0$, we have $\min\{\alpha_{j\pm1}^0\} \ge \frac{p}{q}\abs{\alpha_j^0}$. 
%Moreover, when $j \ge 2$ we also have 
%and  \ge \frac{p}{q}\abs{\alpha_j^0}$.  

\begin{proof} 
Clearly 
\[ 0 \le f(0) = \sum_{k \ge 1} \alpha_k^0 \hatf_k^0(0) = \alpha_1^0 \hatf_1^0(0) = \frac{\alpha_1^0}{2}, \]
thus $\alpha_1^0 \ge 0$. 
Now let $j$ be such that $\alpha_j^0 \le 0$. 
Then, 
\begin{align*}
0 \le f(-2j) 
&= \alpha_{j+1}^0 \hatf_{j+1}^0(-2j) + \alpha_{j}^0 \hatf_{j}^0(-2j) \\
&= \alpha_{j+1}^0 \left(\frac{q}{2}\right) + \alpha_{j}^0 \left( \frac{p}{2} \right) , 
\end{align*}
thus $\alpha_{j+1}^0 \ge -\frac{p}{q}\alpha_j^0 = \frac{p}{q}\abs{\alpha_j^0}$. 
%$-\alpha_j^0 p \le q\alpha_{j+1}^0$, 
%i.e. $\alpha_{j+1}^0 \ge \frac{p}{q} \abs{\alpha_{j}^0}$. 
The same argument applied at $f(2j-2)$ yields
\[
0 \le f(2j-2) = \left(\frac{p}{2}\right)\alpha_j^0 + \left(\frac{q}{2}\right)\alpha_{j-1}^0 ,
\]
thus $\alpha_{j-1}^0 \ge -\frac{p}{q}\alpha_j^0 = \frac{p}{q}\abs{\alpha_j^0}$ as well. 
\end{proof}
\end{lemma}

We now prove the main result. 
\begin{theorem}[Entropic symmetrization resistance of Bernoulli]\label[theorem]{thm:HY_exceeds_HX}
Let $X \sim \Bern(p,a,b)$ with $p \ne \frac{1}{2}$, i.e. $\mathbb{P}(X=b) = p$ and $\mathbb{P}(X=a) = 1-p$. 
Then, any independent symmetrizer $Y$ of $X$ satisfies $H(Y) \ge H(X)$, with equality if and only if $f_Y = f_{-X}$. 
\begin{proof}
By \cref{lem:affine_invariance,cor:psymmetry} it is sufficient to consider $X \sim \Bern(p,-1,1)$ 
%and by \cref{cor:psymmetry} it is sufficient to consider 
with $p > \frac{1}{2}$. 
%Write $f = f_Y \in \calY$. 
Equality is attained when $f_Y=f_{-X}$ by \cref{lem:support_size_one_or_two}; therefore we only need to show that $H(Y) > H(X)$ when $f_Y \ne f_{-X}$, i.e., $\abs{\supp(f_Y)} \ge 3$. 
%Assume $f \ne f_{-X}$ and $H(f) \le H(X)$. 

In this case, by \cref{calYzero_sufficient} either $H(f_Y) > H(X)$ or projecting $f_Y$ onto $\calY^0$ does not decrease its entropy; therefore it suffices to consider %it is sufficient to consider %$f(0) > p$ and
$f_Y \in \calY^0$. %,large_median_at_zero}. 
By \cref{lem:calYzero_representation}, we write $f_Y = \sum_{i=1}^{\infty} \alpha_i^0 \hatf_i^0$ with $\sum_{i=1}^{\infty} \alpha_i^0 = 1$. 
Then, 
$ f(0) = \alpha_1^0 \hatf_1^0(0) = \frac{\alpha_1^0}{2}$, 
so that $\alpha_1^0 > 2p > 1$. 
Note that $\alpha_1^0 = 1 - \sum_{i=2}^{\infty} \alpha_i^0$. 
But, by \cref{lem:calYzero_negative_coefficient_control}, $\sum_{i=2}^{\infty} \alpha_i^0 \ge 0$, so $\alpha_1^0 \le 1$, a contradiction. 
%Thus $H(f) > H(X)$. 
\end{proof}
\end{theorem}

%We remark that the bound in \cref{thm:HY_exceeds_HX} can be approached by a sequence of symmetrizer PMFs not equal to $f_{-X}$. 
%Indeed, define $f_n(x) = \frac{1}{n}f_{-X} + (1-\frac{1}{n})\hatf_1^0(-X) \in \calY$; a straightforward computation shows that $H(f_n) \rightarrow H(f_{-X}) = H(X)$. % as $n \rightarrow \infty$. 

\subsection{Symmetrization on the hypercube}\label[section]{sec:hypercube}
In this section we consider discrete $X,Y \in_R \R^d$.

\begin{notation}
For a nonempty subset $E$ of $[d]$ and vector $z \in \mathbb{R}^d$, denote by $z[E]$ the subvector of $z$ corresponding to the indices in $E$. 
%For a nonempty subset $E = \{i_1, \dots, i_k\}$ of $[d]$ and vector $z \in \R^d$, denote 
%$z[E] = (z_{i_1}, \dots, z_{i_k})$, where $i_j$ are ordered such that $i_j < i_{j+1}$ for all $j \in [k-1]$. 
Also denote $z[E^c] = z[[d] \setminus E]$. % when the latter is defined. 
For a random vector $Z \in_R \mathbb{R}^d$, define random subvectors $Z[E]$ and $Z[E^c]$ similarly. 
\end{notation}

%We begin with an elementary lemma. 
\begin{lemma}\label[lemma]{lem:Rd_projection_works}
Let $Y = (Y_1, \dots, Y_d)$ be an independent symmetrizer for $X = (X_1, \dots, X_d)$ in $\R^d$. 
Then, for any nonempty $E \subseteq [d]$, $Y[E]$ is an independent symmetrizer for $X[E]$. 
In particular, for all $i \in [d]$, $Y_i$ is an independent symmetrizer for $X_i$. 
\begin{proof}
Fix $E = \{ i_1, \dots, i_k \} \subsetneq [d]$. 
%If $E = [d]$ there is nothing to show, so we assume the inclusion is proper. 
%Write $k = \abs{E}$ and fix $z[E]$. 
Clearly, for any $z \in \R^d$
\begin{align*}
f_{X + Y}(z)
&= f_{(X + Y)[E]}(z[E]) f_{(X + Y)[E^c] \mid (X + Y)[E]}(z[E^c] \mid z[E]) .
\intertext{Since $Y$ is an independent symmetrizer of $X$, we have 
$f_{X + Y}(z) = f_{X + Y}(-z)$ and}
f_{(X + Y)[E]}(z[E]) &f_{(X + Y)[E^c] \mid (X + Y)[E]}(u \mid z[E]) \\
&= f_{(X + Y)[E]}(-z[E]) f_{(X + Y)[E^c] \mid (X + Y)[E]}(u \mid -z[E]) .
\intertext{where we have written $u = z[E^c]$. 
Now, write $S = \supp((X + Y)[E^c])$ and sum over $u \in S$:}
%\sum_{u \in S} f_{(X + Y)[E]}(z[E]) &f_{(X + Y)[E^c] \mid (X + Y)[E]}(u \mid z[E]) \\
%&= \sum_{u \in S} f_{(X + Y)[E]}(-z[E]) f_{(X + Y)[E^c] \mid (X + Y)[E]}(u \mid -z[E]) 
%\shortintertext{thus}
 f_{(X + Y)[E]}(z[E]) &\sum_{u \in S} f_{(X + Y)[E^c] \mid (X + Y)[E]}(u \mid z[E]) \\
&= f_{(X + Y)[E]}(-z[E]) \sum_{u \in S} f_{(X + Y)[E^c] \mid (X + Y)[E]}(u \mid -z[E]) .
\end{align*}
%For any fixed $r \in \R^k$, the conditional distribution function $f_{(X + Y)[E^c] \mid (X + Y)[E]}(\cdot \mid r)$ always sums to $1$, so 
Thus
$ f_{(X + Y)[E]}(z[E]) = f_{(X + Y)[E]}(-z[E]) . $
Independence of $X[E]$ and $Y[E]$ is inherited from independence of $X$ and $Y$. 
%Thus $Y[E]$ is an independent symmetrizer of $X[E]$. 
\end{proof}
\end{lemma}

\begin{theorem}\label[theorem]{thm:hypercube_entropy}
Let $X = (X_1, \dots, X_d) \in_R \{\pm1\}^d$ with every $X_i$ asymmetric. 
Let $Y = (Y_1, \dots, Y_d) \in \calY(X)$. 
Then, $H(Y) \ge \frac{1}{d} H(X)$, with equality if and only if the following conditions are all satisfied: 
\begin{enumerate}[(i)]
\item There exists an index $i \in [d]$ such that both $f_{Y_i} = f_{-X_i}$ and $H(Y) = H(Y_i)$,
\item $\abs{p_1-\frac{1}{2}} = \dots = \abs{p_d-\frac{1}{2}}$, 
\item the random variables $X_i$ are mutually independent. 
\end{enumerate}
\begin{proof}
Let $j \in \arg\max_{i \in [d]} H(X_i) = \arg\min_{i \in [d]} \abs{p_i-\frac{1}{2}}$. 
Note that $Y_j$ is an independent symmetrizer of $X_j$ by \cref{lem:Rd_projection_works}. 
Thus we may apply \cref{thm:HY_exceeds_HX} to $X_j$ and $Y_j$. 
Then, 
\begin{align}
H(Y) &\ge H(Y_j) \label[ineq]{ineq:hypercube_entropy_a}\\
&\ge H(X_j) \label[ineq]{ineq:hypercube_entropy_2}\\
%&= \frac{1}{d} \left( d \max_{i \in [d]} H(X_i) \right) \\
&\ge \frac{1}{d} \sum_{i=1}^d H(X_i) \label[ineq]{ineq:hypercube_entropy_3}\\
&\ge \frac{1}{d} H(X) \label[ineq]{ineq:hypercube_entropy_4}
\end{align}
where we used the entropy chain rule in (\ref{ineq:hypercube_entropy_a}), 
\cref{thm:HY_exceeds_HX} in (\ref{ineq:hypercube_entropy_2}), 
and the entropy chain rule and the fact that conditioning reduces entropy in (\ref{ineq:hypercube_entropy_4}). 

Equality is attained in $H(Y) \ge \frac{1}{d} H(X)$ if and only if it is attained in each of (\ref{ineq:hypercube_entropy_a}-
%,ineq:hypercube_entropy_2,ineq:hypercube_entropy_3,
\ref{ineq:hypercube_entropy_4}). 
The equality condition of (\ref{ineq:hypercube_entropy_4}) is $H(X) = \sum_{i=1}^d H(X_i) $
which is in turn equivalent to mutual independence of the $X_i$ (see e.g. \cite[Theorem 1.4(g)]{PW25:book}). 
The equality condition of (\ref{ineq:hypercube_entropy_3}) is exactly (ii). 
Now consider the conditions

\begin{itemize}
\item[(a)] $f_{Y_j} = f_{-X_j}$,
\item[(b)] $H(Y) = H(Y_j)$.
\end{itemize}
Equality is attained in \eqref{ineq:hypercube_entropy_2} exactly when $H(X_j) = H(Y_j)$, which, in light of \cref{thm:HY_exceeds_HX}, is equivalent to (a). 
Equality is attained in \eqref{ineq:hypercube_entropy_a} exactly when (b) is true. 
%Thus $H(Y) = \frac{1}{d} H(X)$ if and only if conditions $(a-b)$ and $(ii-iii)$ are fulfilled. 

We claim that (a-b,ii-iii) is equivalent to (i-iii). 
Clearly (a-b) implies (i); we need only show the converse. 
If (i-iii) is fulfilled, then
$$ H(X_j) = H(X_i) = H(Y_i) = H(Y) \ge H(Y_j) \ge H(X_j) $$
where we have used (in order), condition (ii), condition (i), condition (i) again, the entropy chain rule, and \cref{thm:HY_exceeds_HX}. 
Thus, the last two inequalities are equalities, which forces (b) and also forces $H(Y_j) = H(X_j)$. 
The equality condition of \cref{thm:HY_exceeds_HX} (applied to $X_j$ and $Y_j$) then implies (a). 
\end{proof}
\end{theorem}

\begin{remark}
We do not know whether the equality conditions of \cref{thm:hypercube_entropy} 
can actually happen. 
The issue is the possible nonexistence of ``totally dependent'' symmetrizers $Y$ with $H(Y) = H(Y_i)$ for some index $i \in [d]$. 
On the other hand, symmetrizers $Y \in \calY(X)$ with some dependence between coordinates may be seen to exist by the example $f_Y = \hatf_z$ from \cref{lem:hatf_are_symmetrizers} (for $z \ne 0 \in \mathbb{R}^d$). 
%In particular, we suspect that \cref{thm:hypercube_entropy} can be tightened under the stated assumptions to say that $H(Y) \ge cH(X)$ for some constant or constants $c = c(d) > \frac{1}{d}$. 
\end{remark}

Introducing the additional assumption of independent coordinates $Y_i$ produces the following corollary. 
\begin{corollary}
Let $X = (X_1, \dots, X_d) \in_R \{\pm1\}^d$ with asymmetric coordinates $X_i$. 
Let $Y = (Y_1, \dots, Y_d) \in \calY(X)$ with independent coordinates $Y_1, \dots Y_d$. 
Then, $H(Y) \ge H(X)$, with equality if and only if the following conditions are both satisfied: 

\begin{enumerate}[(i)]
\item For all $i \in [d]$, $f_{Y_i} = f_{-X_i}$,
\item the random variables $X_i$ are mutually independent. 
\end{enumerate}

Also, if $X_1 = \pm X_2 = \dots = \pm X_d$ (with any choice of signs), then $H(Y) \ge dH(X)$ with equality if and only if (i) is satisfied. %for all $i \in [d]$, $f_{Y_i} = f_{-X_i}$. 
\begin{proof}
Clearly 
\[ H(Y) 
%= \sum_{i\in[d]} H(Y_i \mid Y_1, \dots, Y_{i-1}) 
= \sum_{i \in [d]} H(Y_i)
\ge \sum_{i \in [d]} H(X_i)
\ge \sum_{i \in [d]} H(X_i \mid X_1, \dots X_{i-1})
= H(X). \]
%from the entropy chain rule and the inequality in $\mathbb{R}$. 
Equality conditions are clear from the analysis in the proof of \cref{thm:hypercube_entropy}. 
\end{proof}
\end{corollary}

Now we turn to variance symmetrization resistance. 
%Denote the $k \times k$ all-ones matrix by $J_k$. 
%Note that $J_k$ is positive-semidefinite: 
%indeed, $J_k = B^TB$, where $B$ % = [b_{ij}] \in \mathbb{R}^{k \times k}$ 
%is the $k \times k$ matrix having all entries in the first row equal to one, and all other entries zero. 

\begin{lemma}\label[lemma]{lem:psd_diagonal_dominance}
Let $A = [a_{ij}] \in \mathbb{R}^{d \times d}$ be positive semidefinite. 
Then, for all $i, j \in [d]$, 
\begin{equation} \abs{a_{ij}} \le  \frac{1}{2}(a_{ii} + a_{jj}) =  \frac{1}{2}(\abs{a_{ii}} + \abs{a_{jj}}). \label[ineq]{ineq:psd_diagonal_dominance} \end{equation}
Moreover, for $i \ne j$, the following conditions are equivalent: 
\begin{enumerate}[(i)]
\item Equality is attained in (\ref{ineq:psd_diagonal_dominance}). 
\item  The $2 \times 2$ principal submatrix $A[\{i,j\}]$ is not positive definite. 
\item The $2 \times 2$ principal submatrix $A[\{i,j\}]$ is not full rank. 
\item $a_{ii} = a_{jj} = \abs{a_{ij}}$. 
\end{enumerate}

\begin{proof}
We first note that for all $i \in [d]$, $e_i^T A e_i = a_{ii} \ge 0$ 
(with any equality implying $A$ is not positive definite), 
so $a_{ii} = \abs{a_{ii}}$. 
We compute 
\[
0 \le (e_i - e_j)^T A (e_i - e_j) 
%= a_{ii} - a_{ji} - a_{ij} + a_{jj} 
= a_{ii} + a_{jj} - 2a_{ij} 
\]
and similarly
\[
0 \le (e_i + e_j)^T A (e_i + e_j) 
= a_{ii} + a_{jj} + 2a_{ij} 
\]
so that $2\abs{a_{ij}} \le a_{ii} + a_{jj}$. 

Finally, when $i \ne j$, equality in \eqref{ineq:psd_diagonal_dominance} implies $(e_i - e_j)^T A (e_i - e_j) = 0$ or $(e_i + e_j)^T A (e_i + e_j) = 0$, hence $A[\{i,j\}]$ is not positive-definite and therefore not full rank. 
%hence $a_{ii} + a_{ij} = \pm 2 a_{ij}$. 
Then, $\det(A[\{i,j\}])=0$, so that $\abs{a_{ij}} = \sqrt{a_{ii}a_{jj}}$. 
In turn this forces $2\sqrt{a_{ii}a_{jj}} = a_{ii} + a_{jj}$, which is to say the vector $(a_{ii},a_{jj})$ attains equality in the arithmetic-geometric mean inequality.
Therefore $a_{ii}=a_{jj}$, and in turn $\abs{a_{ij}} = \sqrt{a_{ii}^2} = a_{ii}$ as well. 
%Thus $A[\{i,j\}] = a_{ii} J_2$. 
That (iv) implies (i) is obvious. 
%Conversely, if $A[\{i,j\}]$ is a scalar multiple of $J_2$ we clearly have $2a_{ij} = a_{ii} + a_{jj}$. 
\end{proof}
\end{lemma}

\begin{definition}
We define the matrix norm $\norm{\cdot}_{1,1}$ by 
\[ \norm{A}_{1,1} = \sum_{i=1}^d \sum_{j=1}^d \abs{a_{ij}}  .\]
for $A \in \R^{d \times d}$ \cite[page 341]{HJ12:book}. 
\end{definition}
%The functional $\norm{\cdot}_{1,1}$ is a matrix norm 

Denote the trace of a matrix $A$ by $\tr(A)$. 

\begin{lemma}\label[lemma]{lem:d_tr_dominates_oneone_norm}
Let $A \in \R^{d \times d}$ be positive semidefinite. 
Then $d\cdot\tr(A) \ge \norm{A}_{1,1}$, with equality iff every entry of $A$ has the same absolute value. 

\begin{proof}
By \cref{lem:psd_diagonal_dominance},
%Since $A$ is real symmetric, we can bound the sum of the absolute values of the off-diagonal entries in $A = [a_{ij}]$ is bounded: 
%\[{\color{blue}
%\sum_{i=1}^d\sum_{j\ne i} \abs{a_{ij}} 
%\le \frac{1}{2}\sum_{i=1}^d\sum_{j\ne i} (a_{ii} + a_{jj}) 
%= \sum_{i=1}^d\sum_{j\ne i} a_{ii} 
%= (d-1) \sum_{i=1}^d a_{ii} , }\]
\[
\sum_{\{i,j\} \in \binom{[d]}{2}} \abs{a_{ij}} 
\le \frac{1}{2}\sum_{\{i,j\} \in \binom{[d]}{2}} (a_{ii} + a_{jj}) 
= \sum_{\{i,j\} \in \binom{[d]}{2}} a_{ii} 
= \sum_{i=1}^d\sum_{j\ne i} a_{ii} 
= (d-1) \sum_{i=1}^d a_{ii} .
\]
By the same lemma, equality is attained if and only if $\abs{a_{ij}} = a_{ii} = a_{jj}$ for every $\{i,j\} \in \binom{[d]}{2}$. 

Finally, by the nonnegativity of the diagonal entries $a_{ii}$ we have 
\[ \norm{A}_{1,1} \le \left( (d-1) \sum a_{ii} \right) + \left(\sum a_{ii} \right) = d\cdot \tr(A),\]
with equality if and only if $A$ is a nonnegative scalar multiple of $J_d$. 
\end{proof}
\end{lemma}

\begin{theorem}\label{thm:hypercube_variance}
Let $Y = (Y_1, \dots, Y_d)$ and $X = (X_1, \dots, X_d)$ be as in \cref{thm:hypercube_entropy}. 
Then, $\norm{\Cov(Y)}_{1,1} \ge \frac{1}{d} \norm{\Cov(X)}_{1,1}$
%with equality only if the $Y_i$ are uncorrelated. 
with equality if and only if the following conditions are all satisfied:
\begin{enumerate}[(i)]
\item[(i)] The $Y_i$ are uncorrelated.
\item[(ii)] For all $i \in [d]$, $f_{Y_i} = f_{-X_i}$. 
\item[(iii)] $\Cov(X) = \Var(X_i)J_d$, for all $i \in [d]$, where $J_d$ is a p.s.d. matrix with all entries in $\{\pm1\}$. 
\end{enumerate}
\begin{proof}
We compute
\begin{align}
\norm{\Cov(Y)}_{1,1} 
&\ge \tr(\Cov(Y)) \label[ineq]{ineq:hypercube_covariance_1}\\
%&= \sum_i \Var(Y_i) \\
&\ge %\sum_i \Var(X_i) 
%= 
\tr(\Cov(X)) \label[ineq]{ineq:hypercube_covariance_2}\\
&\ge \frac{1}{d} \norm{\Cov(X)}_{1,1}. \label[ineq]{ineq:hypercube_covariance_3}
\end{align}
\Cref{ineq:hypercube_covariance_1} %follows from neglecting the off-diagonal terms in the sum and 
attains equality exactly when the off-diagonal entries in $\Cov(Y)$ are all zero, i.e., the $Y_i$ are uncorrelated. 
\Cref{ineq:hypercube_covariance_2} follows from $d$ applications of \cref{thm:VarY_exceeds_VarX} and the fact that $\tr(\Cov(Y)) = \sum_i \Var(Y_i)$ (and similar for $X$); 
thus equality is attained in \eqref{ineq:hypercube_covariance_2} exactly when 
%it is tight exactly when equality is attained in each of the $d$ applications of \cref{thm:VarY_exceeds_VarX}, i.e., when 
$f_{Y_i} = f_{-X_i}$ for all $i \in [d]$. 
Finally, \eqref{ineq:hypercube_covariance_3} follows from \cref{lem:d_tr_dominates_oneone_norm}. 
It is possible for equality to be attained in \eqref{ineq:hypercube_covariance_3}; one example is $X_1 = \dots = X_d$. 
\end{proof}
\end{theorem}

\begin{remark}
\cref{thm:hypercube_entropy,thm:hypercube_variance} may be seen to apply to $X \in_R S$ where $S \subsetneq \{\pm1\}^d$. 
Thus both inequalities may be made to apply to, e.g., the vertex set of
%of any polytope embedded in the hypercube, e.g. 
a $d$-dimensional simplex embedded in the hypercube. %$\Delta_d \subsetneq \{\pm1\}^d$ in $\mathbb{R}^d$. 
In the case of \cref{thm:hypercube_entropy}, this also applies to invertible affine transformations of $S$ (\cref{lem:affine_invariance}). 

It seems likely that \cref{thm:hypercube_entropy,thm:hypercube_variance} can be strengthened under additional restrictions of this type.
\end{remark}

\begin{remark}
The equality conditions of \cref{thm:hypercube_entropy,thm:hypercube_variance} 
exhibit a curious antisymmetry: 
equality is attained in \cref{thm:hypercube_entropy} roughly when the coordinates of $Y$ are totally dependent and the coordinates of $X$ are totally independent, 
but equality is attained in \cref{thm:hypercube_variance} when the dependence and independence are reversed.  
We suggest that a careful analysis of the behavior of the covariance matrices under different matrix norms could shed light on this situation. 
\end{remark}

We conclude with the following, which shows that the hypothesis that (an appreciable fraction of) the \emph{coordinates} $X_i$ be asymmetric is essential to \cref{thm:hypercube_entropy,thm:hypercube_variance}.
\begin{proposition}
For every $c > 0$ and integer $d \ge 2$ there exists an asymmetric $X = (X_1, \dots, X_d) \in_R \{\pm1\}^d$ and an independent symmetrizer $Y = (Y_1, \dots, Y_d)$ of $X$ such that 
$H(Y) < cH(X)$. 
Moreover, there exist $X'$ and $Y'$ satisfying the same assumptions such that
$\Var(Y') < c\Var(X')$.

\begin{proof}
Clearly we may assume $c<1$. 
Take $X_1 \sim \Bern(p)$ where $p$ is close enough to zero (or one) such that $H(X_1) = c$. 
Take $Y_1 \sim \Bern(1-p)$. 
Take $X_2, \dots, X_d \sim \Bern(\frac{1}{2})$ and $Y_2 = \dots Y_d = 0$ 
(in the case of entropy we may also take $X_2, \dots, X_d \sim \Bern(p,-1,1)$ where $p=\pm1$ and $Y_j = -\mathbb{E}X_j$ for $j \ge 2$). 
Take $X_1$ and $X_2$ to be independent.
Then 
\[ H(Y) = H(Y_1) %Y_2 = 0 always
= c %f_{Y_1} = f_{-X_1}
< c(c+1) %c > 0
= c(H(X_1) + H(X_2)) %construction
= cH(X_1,X_2) %independence
\le cH(X). %dropping terms in entropy chain rule
\]
If we instead choose $p$ such that $\Var(X_1) = c$, and assume only that $X_1$ and $X_2$ are uncorrelated (rather than independent) then the same construction produces
\begin{align*}
\norm{\Cov(Y)}_{1,1} = \Var(Y_1) %Y_2 = 0 always
= c %f_{Y_1} = f_{-X_1}
< c(c+1) %c > 0
&= c(\Var(X_1) + \Var(X_2)) \\ %construction
&= c\norm{\Cov(X_1,X_2)}_{1,1} %X_1, X_2 uncorrelated
\le c\norm{\Cov(X)}_{1,1}. %dropping terms in 1,1 norm 
\qedhere
\end{align*}
\end{proof}
\end{proposition}

\section{Symmetrization in compact groups}\label{sec:groups}

\subsection{Compact groups}

We may also consider symmetrization in a group $G$ together with the Haar measure $\mu$; 
we do not require a specific normalization of $\mu$. 
Probability densities on $G$ are then defined with respect to $\mu$. 
The variance of a $G$-valued random variable is, in general, not meaningful (the elements of $G$ need not be numbers), so we exclusively consider entropic symmetrization resistance with respect to the entropy 
\[ h(f) \coloneqq -\int_G f(w) \log f(w) d\mu(w) \]
for probability densities $f$ on $G$. 
When $G$ is explicitly assumed to be finite, we will continue to write $H$ for the usual Shannon entropy on $G$, which coincides with this definition. 

Denoting $\mathcal{P}(G)$ for the space of probability densities with respect to $\mu$. 
When $G$ is nonabelian we note the existence of two spaces of symmetrizer distributions:
\[ \calY_L^G = \{ g \in \mathcal{P}(G) \mid \forall z \in G \quad (g \convolve f)(z^{-1}) = (g \convolve f)(z) \}\phantom{.} \]
and
\[ \calY_R^G = \{ g \in \mathcal{P}(G) \mid \forall z \in G \quad (f \convolve g)(z^{-1}) = (f \convolve g)(z) \} . \]
Thus, for example, the triviality ``\!$f \in \calY(g)$ iff $g \in \calY(f)$'' becomes ``\!$f \in \calY_R^G(g)$ iff $g \in \calY_L^G(f)$'' in the nonabelian context. 
However, the two spaces have the following relationship. 
\begin{lemma}
Let $G$ be compact with left Haar measure $\mu$ and let $G^{\op}$ denote the opposite group of $G$. 
Take the right Haar measure on $G^{\op}$ to be $\mu$. 
Then, 
\[ \calY_L^G(f) = \calY_R^{G^{\op}}(f) , \]
where $f$ is any probability density on $G$. 

\begin{proof}
Denote the group operation on $G$ by $\cdot$ and on $G^{\op}$ by $\cdot^{\op}$. 
Note that for any $\mu$-measurable $A \subseteq G^{op}$ and $x \in G^{op}$, 
\[ \mu(A\cdot^{\op} x) 
= \mu(\{ a \cdot^{\op} x \mid a \in A \})
= \mu(\{ xa \mid a \in A \})
= \mu(x\cdot A) 
= \mu(A)
\]
So indeed $\mu$ is a right Haar measure on $G^{\op}$. 

Now take $f,g$ to be probability densities with respect to $\mu$ in $G$. 
Since $G$ and $G^{\op}$ have the same topology, $f,g$ are also probability densities on $G^{\op}$. 
Thus, 
\begin{align*}
\calY_R^{G^{\op}}(f)
&= \{ g \mid \forall y \quad (f \convolve^{\op} g)(y) 
                           = (f \convolve^{\op} g)(y^{-1})  \} \\
&= \left\{ g \mid \forall y \quad \int_{G^{\op}} f(w) g(w^{-1}y) \,d\mu(w) 
                           = \int_{G^{\op}} f(w) g(w^{-1}y^{-1}) \,d\mu(w) \right\} \\
%&= \left\{ g \mid \forall y \quad \int_{G^{\op}} f(ys^{-1}) g(s) \,d\mu(s) 
%                           = \int_{G^{\op}} f(y^{-1}s^{-1}) g(s) \,d\mu(s) \right\} \\
&= \left\{ g \mid \forall y \quad \int_{G} g(s) f(ys^{-1}) \,d\mu(s) 
                           = \int_{G} g(s) f(y^{-1}s^{-1}) \,d\mu(s) \right\} \\
&= \{ g \mid \forall y \quad (g \convolve f)(y) 
                           = (g \convolve f)(y^{-1})  \} \\
&= \calY_L^{G}(f) .\qedhere
\end{align*}
\end{proof}
\end{lemma}
Thus we define $\calY \coloneqq \calY_L^G$ for nonabelian $G$. %, and note that any reader interested in $\calY_R$ may consider the result for $\calY_L$ and the opposite group of $G$. 

Additional peculiarities emerge in the group context. 
First, the case $G=\mathbb{Z}_2^n$ ($n \ge 1$) is trivial because all non-identity elements are involutions; that is, every probability distribution on $\mathbb{Z}_2^n$ is symmetric. 
Also, in the case $G=\mathbb{Z}_p$ where $p$ is a prime, multiplicative and additive inverses exist and we may apply \cref{lem:affine_invariance}; thus we may be tempted to proceed along the lines of \cref{sec:reals}. 
However, note that when $G$ is finite, there is hope of enumerating the extreme points and applying the Krein-Milman theorem, an attractive approach which was off limits to us in the case of the reals. 
We thus begin with some general observations, beginning with the existence of certain convenient symmetrizers. 
\begin{lemma}\label[lemma]{lem:f_- and u are symmetrizers}
Let $G$ be a compact group with Haar measure $\mu$. 
Writing $f_-(\cdot) = f(\cdot^{-1}) \in \mathcal{P}(G)$, and $u = \frac{1}{\mu(G)} \in \mathcal{P}(G)$, we have $\{f_-,u\} \subseteq \calY(f)$. 

\begin{proof}
We observe that 
\begin{align*}
\calY = \calY_L 
&= \{ g \in \mathcal{P}(G) \mid \text{for all $z \in G$, } (g \convolve f)(z) = (g \convolve f)(z^{-1}) \} \\
&= \{ g \in \mathcal{P}(G) \mid \text{for all $z \in G$, } \int_{G} g(w) (f(w^{-1}z) - f(w^{-1}z^{-1}) ) \,d\mu(w) = 0 \} ,
\end{align*}
thus clearly %$g = u$ is a solution of the equations defining $\calY$ and
$u \in \calY$. For $f_-$, compute
\begin{align*}
(f_- \convolve f)(z^{-1}) 
&= \int_G f_-(w) f(w^{-1}z^{-1}) \,d\mu(w) \\
&= \int_G f_-(zw)f(w^{-1}) \,d\mu(w) \\
&= \int_G f_-(u)f(u^{-1}z) \,d\mu(u) %\\ % zw = u, so w^{-1} = u^{-1}z
= (f_- \convolve f)(z). \qedhere
\end{align*}
\end{proof}
\end{lemma}

\begin{theorem}\label{thm:no_symmetrizers_in_interior}
Let $G$ be a compact group with Haar measure $\mu$ and let $f$ be a probability density with respect to $\mu$. 
If $\inf_{x \in G} f(x) > 0$, 
then $f$ is \emph{not} symmetrization resistant in $G$. 

\begin{proof}
By \cref{lem:f_- and u are symmetrizers}, $f_-$ and the uniform $u$ are both in $\calY(f)$. 
Without loss of generality we take $\mu(G)=1$. 

In the case $h(f) = h(u)$, we have $h(u) - h(f) = \DKL{f}{u} \ge 0$, with equality if and only if $f = u$ $\mu$-a.e.. 
Thus $f \convolve g = u \convolve g = u$ for any probability density $g$ on $G$, so this $f$ is not symmetrization resistant. 
Thus we assume $h(f) < h(u)$. 

Define $\varepsilon \coloneqq \inf f_- = \inf f  \in (0,1)$ and $b = \frac{f_- - \varepsilon u}{1-\varepsilon}$. 
Since $b \ge 0$, $b$ is a probability measure, and since the symmetry equations are linear, $b \in \calY(f)$. 
Thus $f_- = (1-\varepsilon)b + \varepsilon u$. 
Applying concavity of entropy, 
\begin{align*} h(f) 
= h(f_-) 
= h(\varepsilon u + (1-\varepsilon) b)
&\ge \varepsilon h(u) + (1-\varepsilon) h(b) \\
&> \varepsilon h(f) + (1-\varepsilon) h(b) 
\end{align*}
so that $h(f) > h(b)$. 
\end{proof}
\end{theorem}

Note that, when $G$ is finite, $\inf f$ is a minimum and the above theorem says that there are no symmetrization-resistant distributions in the interior of the probability simplex.

%The next lemma shows that, in order to check symmetrization resistance for a distribution $f$ on a group $H$ embedded in a larger group $G$, it is necessary to consider the symmetrizers of $f$ in $H$. 
We now turn our attention the behavior of the symmetrization problem with respect to subgroups. 
\begin{lemma}
Let $f$ be a probability density on $G$ such that $\supp(f)$ is contained in some proper subgroup $H$ of $G$ with $\mu(H) > 0$ and Haar measure $\mu_H$ defined to be the Haar measure $\mu$ of $G$ restricted to $H$. 

Then if $g'$ is a symmetrizer for $\restr{f}{H}$ (the restriction of $f$ to $H$) in $H$, 
\[ g \coloneqq \begin{cases} g'(x) & x \in H \\
0 & \text{otherwise}
\end{cases} \]
is a symmetrizer of $f$ in $G$. 

\begin{proof}
Since $g'$ is a symmetrizer for $\restr{f}{H}$, 
\begin{align*}
(g \convolve f)(x) 
&= \int_{G} g(r)f(r^{-1}x) \,d\mu(r)\\ %integrand is zero when $x \not\in H$ by defn of $g$
&= \int_{H} g(r)f(r^{-1}x) \,d\mu(r)
= (g' \convolve \restr{f}{H})(x) 
= (g' \convolve \restr{f}{H})(x^{-1}) 
%= \dots 
= (g \convolve f)(x^{-1}). \qedhere
\end{align*}
\end{proof}
\end{lemma}

We also note that, if $G$ is a direct sum of two smaller groups, then symmetrization of product densities on $G$ is at least as easy as (and symmetrization resistance is at least as hard as) symmetrization on the smaller groups. 
Here the independence assumption $f = f_H f_I$ plays a similar role to the independence assumptions in the results in \cref{sec:hypercube}. 
\begin{lemma}\label[lemma]{lem:direct_sum_lifting}
Take $G = H \otimes I$ with respective Haar measures $\mu_G$, $\mu_H$, and $\mu_I$ such that $\mu_G = \mu_H \otimes \mu_I$. 
%Further let the measures $\mu_G,\mu_H,\mu_I$ satisfy
%$\mu_G(S \times T) = \mu_H(S) \mu_I(T)$ for all $\mu$-measurable $S \times T$. 
Let $f_H$ and $f_I$ be probability densities with respect to $\mu_H$ and $\mu_I$, respectively. 
Then $(g_H,g_I) \in \calY(f_H) \times \calY(f_I)$ implies $g_Hg_I \in \calY(f_Hf_I)$.  
\begin{proof}
For all $a \in H$, $(g_H \convolve f_H)(a) = (g_H \convolve f_H)(a^{-1})$, and for all $b \in I$, $(g_I \convolve f_I)(b) = (g_I \convolve f_I)(b^{-1})$, so
\[
(g_H \convolve f_H)(a)(g_I \convolve f_I)(b)
= (g_H \convolve f_H)(a^{-1})(g_I \convolve f_I)(b^{-1}). \]
Applying Fubini-Tonelli,
%\int_H g_H(r)f_H(r^{-1}a)\,d\mu_H(r) \int_I g_I(s)f_I(s^{-1}b) \,d\mu_I(s)
%&= \int_H g_H(r)f_H(r^{-1}a^{-1})\,d\mu_H(r) \int_I g_I(s)f_I(s^{-1}b^{-1}) \,d\mu_I(s) \\
\begin{multline*}
\int_{G} g_H(r)g_I(s) f_H(r^{-1}a)f_I(s^{-1}b) \,d\mu_G((r,s)) \\
= \int_{G} g_H(r)g_I(s) f_H(r^{-1}a^{-1})f_I(s^{-1}b^{-1}) \,d\mu_G((r,s)) .
\end{multline*}
Thus, writing $f \coloneqq f_H f_I$ and $g \coloneqq g_H g_I$, 
\[
\int_G g(r,s) f(r^{-1}a,s^{-1}b) \,d\mu_G((r,s))
= \int_G g(r,s) f(r^{-1}a^{-1},s^{-1}b^{-1}) \,d\mu_G((r,s)) ,
\]
so $(g \convolve f)(a,b) = (g \convolve f)(a^{-1},b^{-1})$. % for all $(a,b)$. 
\end{proof}
\end{lemma}

The next example shows that while symmetrization resistance of $f_H$ on $H$ and $f_I$ on $I$ is necessary for the product $f(x,y) = f_H(x)f_I(y)$ to be symmetrization resistant on $H \oplus I$, it is not sufficient. 
\begin{example}
Consider the PMF $f(x,y) = f_1(x)f_1(y)$ on $\mathbb{Z}_3^2$, where $f_1(p) = (1-p,p,0)$ is symmetrization resistant on $\mathbb{Z}_3$ for $p \in (0,\frac{1}{3}) \cup (\frac{2}{3},1)$. 
Then, a straightforward computation shows that $g(x,y) = \frac{1}{3}\mathbb{1}_{x=y} \in \calY(f)$, and it is also readily verified that
\[ H(f) = -2(p \log p + (1-p)\log(1-p)) > \log_2(3) = H(g) \]
for $p \in (0.238468, 0.761532)$. 
%Let's check: writing $h \coloneqq 3 f \convolve g$ (to take advantage of the fact that $g=1/3$ whenever it is nonzero) 
%\[ h(-1,-1) = f(-1,-1) + f(0,0) = q^2 + p^2 = f(-1,-1) + f(0,0) = h(1,1) \]
%\[ h(-1,0) = f(-1,0) + f(0,1) + f(1,-1) = f(-1,0) = pq = f(0,-1) = f(0,-1) + f(1,0) + f(-1,1) = h(1,0) \]
%\[ h(0,-1) = f(0,-1) + f(1,0) + f(-1,1) = f(0,-1) = pq = f(-1,0) = f(0,1) + f(1,-1) + f(-1,0) = h(0,1) \]
%\[ h(1,-1) = f(1,-1) + f(-1,0) + f(0,1) = f(-1,0) = pq = f(0,-1) = f(-1,1) + f(0,-1) + f(1,0) = h(-1,1) \]
Thus, in particular, although by \cref{thm:Z3} $f_1$ is symmetrization resistant for $p \in [1/4,1/3]$, $f$ is \emph{not} symmetrization resistant at these same values of $p$. 
\end{example}

The next example shows that the assumption that $f$ is the product of $f_H$ and $f_I$ is necessary in \cref{lem:direct_sum_lifting}; that is, it is not sufficient merely for the marginals of $f$ to be $f_H$ and $f_I$. 
\begin{example}
Consider $f(i,j) = %0\mathbb{1}_{\{i=j=0\}} + 
p \mathbb{1}_{\{i=j=1\}} + q \mathbb{1}_{\{i=j=-1\}}$ with $p \in (\frac{1}{2},1)$ %(so  $q = 1-p$)
on $G = \mathbb{Z}_3^2$. 
Take $H = I = \mathbb{Z}_3$ and $g_H = g_I = %\mathbb{1}_{\{0\}} + 
q\mathbb{1}_{\{1\}} + p\mathbb{1}_{\{-1\}}$. 
Then 
\begin{align*}
(f \convolve (g_Hg_I))(-1,-1) 
&= \sum_{i,j} f(i,j) g_H(-1-i) g_I(-1-j) \\
%&= \sum_{i} f(i,i) g_H(-1-i) g_I(-1-i) \\
&= \sum_{i} f(i,i) g_H(-1-i)^2 \\
&= g_H(-1)^2f(0,0) + g_H(1)^2f(1,1) + g_H(0)^2f(-1,-1) \\
%&= 0              + q^2p           + 0 \\
&= q^2p .
\intertext{Similarly, }
(f \convolve (g_Hg_I))(1,1) 
&= \sum_{i,j} f(i,j) g_H(1-i) g_I(1-j) \\
%&= \sum_{i} f(i,i) g_H(1-i) g_I(1-i) \\
&= \sum_{i} f(i,i) g_H(1-i)^2 \\
&= g_H(1)^2f(0,0) + g_H(0)^2f(1,1) + g_H(-1)^2f(-1,-1) \\
%&= 0              + q^2p           + 0 \\
&= p^2q .
\end{align*}
Thus, defining the marginals $f_H$ and $f_I$ by $f_H(i) = \sum_{j \in \mathbb{Z}_3} f(i,j)$ and $f_I(j) = \sum_{i \in \mathbb{Z}_3} f(i,j)$, we have $g_H \in \calY(f_H)$ and $g_I \in \calY(f_I)$, but $g_Hg_I \not\in \calY(f)$. 
\end{example}

Finally we note that the symmetrization problem is invariant under automorphisms. 
\begin{lemma}\label[lemma]{lem:automorphisms}
Let $G$ be a compact group with Haar measure $\mu$ and let $\varphi$ be an automorphism of $G$. 
Then
\[ g \in \calY(f) \quad \text{ iff } \quad g \circ \varphi \in \calY(f \circ \varphi). \]

\begin{proof}
Clearly $g \in \calY(f)$ if and only if for all $x \in G$,
\begin{align*}
0 &= \int_G g(r) \big( f(r^{-1}x) - f(r^{-1}x^{-1}) \big) \,d\mu(r) 
\shortintertext{i.e.}
0 &= \int_G g(\varphi(\varphi^{-1}(r))) \big( f(\varphi(\varphi^{-1}(r^{-1}x))) - f(\varphi(\varphi^{-1}(r^{-1}x^{-1}))) \big) \,d\mu(r) ,
%\\ 0 &= \int_G g(\varphi(\varphi^{-1}(r))) \big( f(\varphi(\varphi^{-1}(r^{-1})\varphi^{-1}(x))) - f(\varphi(\varphi^{-1}(r^{-1})\varphi^{-1}(x^{-1}))) \big) \,d\mu(r) ,
\intertext{so that, writing $s = \varphi^{-1}(r)$ and $y = \varphi^{-1}(x)$,}
0 &= \int_G g(\varphi(s)) \big( f(\varphi(s^{-1}y)) - f(\varphi(s^{-1}y^{-1})) \big) \,d\mu(s) . \qedhere
\end{align*}
\end{proof}
\end{lemma}

\subsection{Finite cyclic groups}
In the rest of this section we investigate symmetrization in the groups $G= \mathbb{Z}_n$, where $n \ge 3$ is an integer. 

\begin{lemma}
When $G = \mathbb{Z}_{2q+1}$, the distributions which are symmetric about zero are exactly 
\[ \conv\left(\left\{\mathbb{1}_0, \frac{1}{2}(\mathbb{1}_{-1} + \mathbb{1}_{1}), \dots, \frac{1}{2}(\mathbb{1}_{-q} + \mathbb{1}_{q})\right\}\right) . \]
When $G = \mathbb{Z}_{2q}$, the distributions which are symmetric about zero are exactly 
\[ \conv\left(\left\{\mathbb{1}_0, \mathbb{1}_q, \frac{1}{2}(\mathbb{1}_{-1} + \mathbb{1}_{1}), \dots, \frac{1}{2}(\mathbb{1}_{-q+1} + \mathbb{1}_{q-1})\right\}\right) . \]
\begin{proof}
Immediate.
%Clearly all such distributions are symmetric about zero; conversely, $f \in_R \mathbb{Z}_{2q+1}$ is symmetric about zero iff $f(-k) = f(k)$ for all $k \in [q]$; thus $f$ lies in the described set. 
\end{proof}
\end{lemma}

\begin{lemma}\label[lemma]{lem:Zp_symmetrizer_space}
For $f \in_R \mathbb{Z}_{2q+1}$ ($q \ge 1$),
\[ \calY(f) = 
    \{ g \in_R \mathbb{Z}_{2q+1}
    \mid \forall x \in [q] \quad
    \sum_{i=-q}^q (f(x-i) - f(-x-i)) g(i) = 0 \}.
    \]
In particular, $\dim(\calY(f)) \ge q$. 
\begin{proof}
We identify the elements of $\mathbb{Z}_{2q+1}$ with $[-q,q] \cap \mathbb{Z}$ so that inverses are obvious. 
By direct computation, $\calY(f)$ is the set of random vectors $g  \in \mathbb{Z}_{2q+1}$  such that for all integers $x \in [-q,q]$, $(f \convolve g)(x) = (f \convolve g)(-x)$, i.e.
\[    \sum_{i=-q}^q (f(x-i) - f(-x-i)) g(i) = 0 \}.    \]
But the condition corresponding to $x=0$ is vacuous, and for all $x \in [q]$ the conditions produced by $x$ and $-x$ are identical. 

To see the dimension, note that the simplex of probability measures on $\mathbb{Z}_{2q+1}$ has dimension $2q$, and we have imposed $q$ linear constraints. 
\end{proof}
\end{lemma}

%\subsection[The cyclic group of order 3]{$\mathbb{Z}_3$}
\subsection[Exact characterization in the cyclic group of order 3]{Exact characterization of symmetrization resistance in $\mathbb{Z}_3$}

\begin{theorem}\label{thm:Z3}
The symmetrization resistant distributions in $G = \mathbb{Z}_3$ (with addition modulo three) are exactly the extreme two-thirds of every side of $\Delta_3$, except the corners, which are already symmetric (\cref{fig:Z3}). 

\begin{figure}
\centering
\begin{tikzpicture}[scale=5]

% Coordinates of equilateral triangle
\coordinate (A) at (0,0);
\coordinate (B) at (1,0);
\coordinate (C) at (0.5,{sqrt(3)/2});

% Triangle centroid
\coordinate (U) at ($ (A)!.333!(B) + (A)!.333!(C) $);

% Draw sides with thick lines
\draw[very thin] (A) -- (B);
\draw[very thin] (B) -- (C);
\draw[very thin] (C) -- (A);

% Mark middle thirds thinner
\foreach \P/\Q in {A/B,B/C,C/A} {
  \coordinate (M1) at ($ (\P)!1/3!(\Q) $);
  \coordinate (M2) at ($ (\P)!2/3!(\Q) $);
  \draw[very thick,red] (\P) -- (M1);
  \draw[very thick,red] (M2) -- (\Q);
}

% Draw line through centroid U at 20 degrees
\coordinate (L1) at (1,0.5);
\coordinate (L2) at ($ (L1)!2!(U) $);
\coordinate (fminus) at ($ (A)!0.115!(C) $);
\coordinate (f) at ($ (B)!0.115!(C) $);
%\coordinate (f) at ($(1,0)-fminus$);
\draw[dashed] (L1) -- (L2);

% Labels
\node[below left] at (A) {$(1,0,0)$};
\node[below right] at (B) {$(0,0,1)$};
\node[above right] at (f) {$f$};
\fill (f) circle (0.013);
\node[above left] at (fminus) {$f_-$};
\fill (fminus) circle (0.013);
\node[right] at (L1) {$\calY(f)$};
\node[above] at (C) {$(0,1,0)$};
\draw[fill=white] (A) circle[radius=0.013];
\draw[fill=white] (B) circle[radius=0.013];
\draw[fill=white] (C) circle[radius=0.013];
\fill (U) circle (0.01);
\node[below right] at (U) {$u$};

\end{tikzpicture}
\caption{The probability simplex for $\mathbb{Z}_3$. 
The red parts of the simplex are symmetrization resistant. 
A representative symmetrization-resistant PMF $f$ with symmetrizer space $\calY(f)$ is shown. }\label{fig:Z3}
\end{figure}
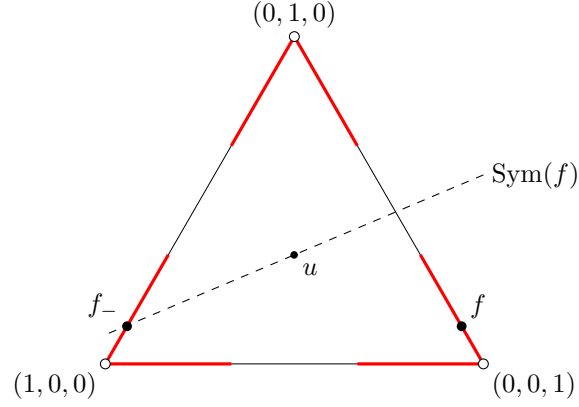

\begin{proof}
Indeed, by \cref{thm:no_symmetrizers_in_interior} it is sufficient to consider $f \in \partial\Delta_3$. 
By \cref{lem:Zp_symmetrizer_space}, $\dim(\calY(f)) \ge 1$, and clearly $\dim(\calY(f)) = 2$ iff $\calY(f) = \Delta_3$, i.e. $f$ is symmetric. 
Since only asymmetric $f$ may be symmetrization resistant, we therefore conclude that $\dim(\calY(f)) = 1$. 
Note as in the proof of \cref{thm:no_symmetrizers_in_interior} that $u = (\frac{1}{3},\frac{1}{3},\frac{1}{3}) \in \calY(f)$ and $f_- \in \calY(f)$. 
Since $f_-$ is on the boundary of the simplex, $u \ne f_-$ and we conclude that $\calY(f) = \aff(\{u,f_-\}) \cap \Delta_3$. 
Thus $\calY(f)$ intersects $\partial \Delta_3$ at exactly two points. 
One of these points is $f_-$; denote the other by $p$. 
Thus for $g \in \calY(f) = \conv(f_-,p)$ there exists $\lambda \in [0,1]$ such that $g = \lambda p + (1-\lambda)f_-$ and 
\[ H(g) \ge \lambda H(p) + (1-\lambda)H(f_-) ,\]
In turn, $H(f_-) = H(f)$ and $f$ is not symmetrization resistant if and only if $H(f) \ge \lambda H(p) + (1-\lambda)H(f_-)$, i.e. if and only if $H(f) \ge H(p)$. 
Note that since $f$ and $p$ are in $\partial\Delta$, each must have an entry equal to zero, i.e., both $f$ and $p$ are two-point distributions, albeit with different support. 
Explicitly, writing $a \le \frac{1}{2}$ and $1-a$ for the nonzero elements of $f$, a routine computation shows that the nonzero elements of $p$ are $\frac{1-2a}{2-3a}$ and $\frac{1-a}{2-3a}$. 
Thus $H(f) \le H(p)$ if and only if $1-a \ge \frac{1-a}{2-3a}$, 
%(see \cref{unique_large_mass} for relevant properties of Bernoulli entropy function)
%i.e. $2-3a \ge 1$,
%i.e. $3a \le 1$,
i.e. $a \le \frac{1}{3}$. 
Thus the set of symmetrization-resistant distributions in $\mathbb{Z}_3$ is exactly the extreme two-thirds of each side of the boundary of $\Delta_3$, except for the distributions which are already symmetric (in this case, just $\mathbb{1}_0$). 
\end{proof}
\end{theorem}

%\subsection[The cyclic group of order 4]{$\mathbb{Z}_4$}
\subsection[Exact characterization in the cyclic group of order 4]{Exact characterization of symmetrization resistance in $\mathbb{Z}_4$}

\begin{theorem}\label{thm:Z4}
The symmetrization resistant distributions on $\mathbb{Z}_4$ are exactly the two-point distributions supported on $\{0,1\}$, $\{1,2\}$, $\{2,3\}$, and $\{3,0\}$, as shown in \cref{fig:Z4 symmres}.

\begin{proof}
Let $f$ be a PMF on $\mathbb{Z}_4$ with support $F$. 
If $\abs{F}=1$, $f$ is symmetric. 
If $\abs{F}=4$, $f$ is not symmetrization resistant by \cref{thm:no_symmetrizers_in_interior}. 

If $\abs{F}=2$, Note that $F \coloneqq \supp(f)$ cannot be a coset of the subgroup $\{0,2\}$, 
otherwise there is a constant $c$ such that $F - \{c\} = \{0,2\}$ and
$f$ is symmetric. 
Since we do not have multiplicative inverses for all the nonzero elements in $\mathbb{Z}_4$, we do not have affine invariance; 
however, we do still have translation invariance ($g \in \calY(f)$ iff $g(\cdot +c) \in \calY(f(\cdot -c))$) 
and, by \cref{lem:automorphisms}, invariance under additive inverses ($g \in \calY(f)$ iff $g_- \in \calY(f_-)$). 
Thus the only example we need to consider is $F = \{0,1\}$ with $p \coloneqq f(1) \ge \frac{1}{2}$. 
Since $\calY(f)$ is formed by the intersection of a hyperplane with $\Delta_4$, $\dim(\calY(f)) \le 2$ and the elements $g$ of $\Ex(\calY(f))$ must have at most two nonzero entries; that is, for $G \coloneqq \supp(g)$, $\abs{G} \le 2$. 
Since there are no constants which symmetrize $f$, we consider the cases where $\abs{G} = 2$. 
When $G = \{0,1\}$ or $G = \{2,3\}$ it is readily observed that there is no solution.
When $G = \{0,2\}$ or $G = \{1,3\}$ the unique solution is $g(x) = \frac{1}{2}$ for $x \in G$; that is, $H(g) = 1 \ge H(f)$. 
When $G = \{0,3\}$ the unique solution is $g = f_-$, and
when $G = \{1,2\}$ the unique solution is $g = g^*$ with $g^*(1) = p$, $g^*(2) = 1-p$, $g^*(0) = g^*(3) = 0$. 
%Thus $g^* \ne f_-$ is a symmetrizer of $f$ with $H(g^*) = H(f)$. 
Therefore $f$ is symmetrization resistant as long as its support is not contained in a coset of $\{0,2\}$. %a proper subgroup of $\mathbb{Z}_4$.  

It remains only to consider the case $\abs{F} = 3$. 
By invariance under shifts, we may take $F = \{0,1,2\}$. 
PMFs with this support exist in the triangle depicted in \cref{fig:Z4_012}. 
Note that for any such $f$ we have $\{\frac{1}{2}\mathbb{1}_{\{0,2\}},\frac{1}{2}\mathbb{1}_{\{1,3\}}\} \subseteq \calY(f)$. 
We will treat in turn the elements of $A$,$B$, $C$, and $C'$ supported on $3$ points. 

The central region $B$ is the simplest: since $\Ex(B) = \{\frac{1}{2}\mathbb{1}_{\{0,1\}},\frac{1}{2}\mathbb{1}_{\{0,2\}},\frac{1}{2}\mathbb{1}_{\{1,2\}}\}$ consists of two-point distributions with entropy $1$, any $f \in B \setminus \Ex(B)$ is a nontrivial convex combination of the extreme points and therefore satisfies $H(f) > 1 = H(\frac{1}{2}\mathbb{1}_{\{0,2\}})$.
Therefore $\frac{1}{2}\mathbb{1}_{\{0,2\}}$ is an entropy-reducing symmetrizer and $f$ is not symmetrization resistant.  
%The bottom corner $\frac{1}{2}\mathbb{1}_{\{0,2\}}$ is already symmetric, hence not symmetrization resistant. 
Note that we treated the edges of $B$, and the remaining edges of $A$, $C$, and $C'$ consisted of two-point distributions, so we now need only treat the interiors $A^{\circ}$, $C^{\circ}$, and $(C')^{\circ}$. 

In the case $f = (1-a-b, a, b, 0) \in C^{\circ}$, we have $a+b < \frac{1}{2}$. 
The unique nontrivial symmetry equation satisfied by $g = (g_0, g_1, g_2, g_3) \in \calY(f)$ is 
\[ g_1f_2 + g_2 f_1 + g_3 f0= (f \convolve g)(-1) = f \convolve g)(1) = g_0 f_1 + g_1 f_0 + g_3 f_2, \]
which reduces to
\[ g_3(1-a-2b) + g_2a = g_1(1-a-2b) + g_0a. \]
In this last equation, every term is nonnegative, so we readily see that there are symmetrizers (which are necessarily extreme points) supported in $\{0,2\}$, $\{0,3\}$, $\{1,2\}$, $\{1,3\}$ (since the intersection of a hyperplane with the simplex $\Delta_4$ can have at most four corners, these are actually all of the extreme points of $f$, though we will not use this fact). 
The extreme point with support $\{1,2\}$ is $g = (0, \frac{a}{1-2b}, \frac{1-a-2b}{1-2b}, 0)$. 
Writing $H_B(p)$ for the entropy of a two-point random variable with parameter $p$, we have that 
\[ H(g) = H_B\bigg(\frac{1-a-2b}{1-2b}\bigg) \le H_B(1-a-2b) 
= H_B(f_0) < H(f), \]
thus $f$ is not symmetrization resistant. 

Symmetrizing $f = (b,a,1-a-b,0) \in C'$ (so $a + b \le \frac{1}{2}$) is equivalent to symmetrizing $f_- = (b,0,1-a-b,a)$ by invariance under reflection.
In turn, symmetrizing $f_-$ is equivalent to symmetrizing $f' = f_-(\cdot +2) = (1-a-b,a,b,0) \in C$. 
Thus $f \in (C')^{\circ}$ implies $f$ is not symmetrization resistant. 

Finally, consider $f = (a,1-a-b,b,0) \in A^{\circ}$. 
The argument is similar to $C^{\circ}$. 
Here we have $a + b \le \frac{1}{2}$, and by invariance under shifts and reflexions it is equivalent to symmetrize 
%$f(-\cdot) = (a,0,b,1-a-b)$ and
$f(-\cdot + 2) = (b,1-a-b,a,0)$, so without loss of generality we may take $b \ge a$. 
In the case $a=b$ there is nothing to show because $f$ is symmetric, so we may further assume $b > a$.  
The unique symmetry equation for $g \in \calY(f)$ is 
\[ 
%g_0f_3 + g_1f_2 + g_2f_1 + g_3f_0 = 
g_1f_2 + g_2f_1 + g_3f_0 = 
= (f \convolve g)(-1)
= (f \convolve g)(1)
= g_0f_1 + g_1f_0 + g_3f_2
%= g_0f_1 + g_1f_0 + g_2f_3 + g_3f_2
\]
which is equivalent to 
\[ g_2(1-a-b) + g_1(b-a) = g_0(1-a-b) + g_3(b-a) .\]
%As in the case $C^{\circ}$ 
Again all the terms in this equation are nonnegative. 
In particular, %there is a symmetrizer $g$ with support $\{0,1\}$ given by 
$g = (\frac{b-a}{1-2a},\frac{1-b-a}{1-2a},0,0)$ is a solution satisfying 
\[ H(g) 
= H_B\bigg(\frac{1-b-a}{1-2a}\bigg) 
< H_B(1-b-a) 
= H_B(f_1) 
< H(f). \qedhere
\]
%so $f$ is not symmetrization resistant. 
\end{proof}
\end{theorem}

Kagan et al \cite[Theorem 4]{KMSVV99} showed that there exists an asymmetric distribution in $\mathbb{R}$ which has at least two distinct minimum-variance symmetrizers. 
Though their example is not readily adapted to the case of entropy, we obtain the following. 
\begin{theorem}
In $\mathbb{Z}_3$ and $\mathbb{Z}_{4q}$ with $2 \nmid q$, there exist symmetrization resistant random variables for which the minimum-entropy independent symmetrizer is not unique. 

\begin{proof}
In $\mathbb{Z}_3$, writing PMFs as $z = (z(-1),z(0),z(1))$, we see that $f = \frac{1}{3}(1,2,0)$ (which is symmetrization resistant by \cref{thm:Z3}) has $\frac{1}{3}(2,0,1)$ and $\frac{1}{3}(0,2,1)$ as symmetrizers. 

%In $\mathbb{Z}_4$, \cref{thm:Z4} provides that the PMF $f$ with $f(0) = f(1) = \frac{1}{2}$ has equal-entropy symmetrizers $g$ and $h$ defined by $g(0) = g(2) = \frac{1}{2}$ and $h(1) = h(3) = \frac{1}{2}$.  

In $\mathbb{Z}_{4q}$,
$f = \frac{1}{2}(\mathbb{1}_{0} + \mathbb{1}_{q})$ 
has equal-entropy symmetrizers 
$f_{-} = \frac{1}{2}(\mathbb{1}_{0} + \mathbb{1}_{3q})$
and
$g = \frac{1}{2}(\mathbb{1}_{0} + \mathbb{1}_{2q})$. 
%Per \cref{thm:Z4}, in $\mathbb{Z}_4$ these are minimum-entropy symmetrizers. 
%
%When $q$ is not divisible by two, we 
Note that $h \in \calY(f)$ if and only if for all $s \in [4q]$
\begin{equation}\label{eq:symm_4q_special}
h(s) + h(s-q) = h(-s) + h(-s-q) 
\end{equation}
so that if $h(s) > \frac{1}{2}$ for some $s$, we have 
\[ \frac{1}{2} < h(s) \le h(-s) + h(-s-q) \]
and therefore 
\[ 1 < h(s) + h(-s) + h(-s-q) \]
so that either $s = -s$ (so $s = 0$) or $s = -s-q$ (so $2s = q$). 
Since $2$ does not divide $q$, the latter cannot happen and we need only consider $s=0$.
Evaluating \eqref{eq:symm_4q_special} at $s=q$ yields
\[ h(q) + h(0) = h(-q) + h(-2q) \]
so that 
\[ \frac{1}{2} < h(0) \le h(-q) + h(-2q) \]
and since $0$, $-q$, and $-2q$ are distinct modulo $4q$, we obtain 
\[ 1 < h(0) + h(-q) + h(-2q) \le \sum_x h(x) = 1 , \]
a contradiction. 
Thus $h \le \frac{1}{2}$ everywhere and
\[ H(h) = \sum_{x}h(x) \log_2\bigg(\frac{1}{h(x)}\bigg) \ge \sum_{x} h(x) = 1 = H(f), \]
so our $f$ is symmetrization resistant. 
%
%In the case $2 \mid q$, we obtain symmetrizers 
%$\mathbb{1}_{-\frac{q}{2}} \in \calY(f)$ and $\mathbb{1}_{\frac{3q}{2}} \in \calY(f)$ so that our choice of $f$ is not symmetrization resistant. 
\end{proof}
\end{theorem}

\begin{figure}
\centering
\begin{tikzpicture}[scale=5]

% Vértices del triángulo equilátero
\coordinate (P0) at (0,0);
\coordinate (P2) at (1,0);
\coordinate (P1) at (0.5,{sqrt(3)/2});

% Puntos para subdivisiones
\coordinate (M)  at (0.5,0);                % punto medio de la base
\coordinate (H)  at (0.5,{sqrt(3)/6});      % altura de la región B
\coordinate (L)  at (0.25,{sqrt(3)/4});     % extremo izquierdo línea horizontal
\coordinate (R)  at (0.75,{sqrt(3)/4});     % extremo derecho línea horizontal

% Contorno del conjunto
\draw[red, line width=1.2pt] (P0)--(P1)--(P2);
\draw[line width=1.2pt] (P0)--(P2);

% Subdivisiones internas
\draw[dashed] (L)--(R);     % línea horizontal
%\draw[dashed] (P1)--(M);    % línea vertical
\draw[dashed] (L)--(M);     % diagonal izquierda
\draw[dashed] (R)--(M);     % diagonal derecha

% Etiquetas regiones
\node at (0.5,0.58) {$A$};
\node at (0.5,0.30) {$B$};
\node at (0.28,0.22) {$C$};
\node at (0.72,0.22) {$C'$};

% Círculos abiertos en vértices
\fill[white] (P0) circle (0.013);
\draw (P0) circle (0.013);

\fill[white] (P1) circle (0.013);
\draw (P1) circle (0.013);

\fill[white] (P2) circle (0.013);
\draw (P2) circle (0.013);

% Etiquetas de vértices
\node[below left]  at (P0) {$\mathbb{1}_{\{0\}}$};
\node[above]       at (P1) {$\mathbb{1}_{\{1\}}$};
\node[below right] at (P2) {$\mathbb{1}_{\{2\}}$};
\node[above left]  at (L)  {$(\frac{1}{2},\frac{1}{2},0,0)$};
\node[above right] at (R)  {$(0,\frac{1}{2},\frac{1}{2},0)$};
\node[below] at (M)  {$(\frac{1}{2},0,\frac{1}{2},0)$};

\end{tikzpicture}
\caption{The facet of the probability simplex for $\mathbb{Z}_4$ containing the distributions with support contained in $\{0,1,2\}$. 
The red sides are symmetrization resistant by the case $\lvert F\rvert=2$ of \cref{thm:Z4}, and the regions $A$, $B$, $C$ and $C'$ are treated in the case $\lvert F\rvert=3$ of the proof of \cref{thm:Z4}.
 }
\label{fig:Z4_012}
\end{figure}
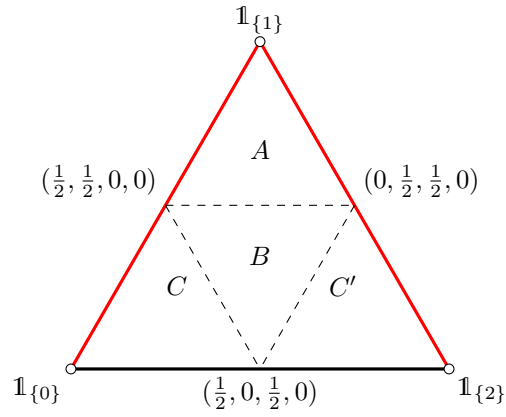

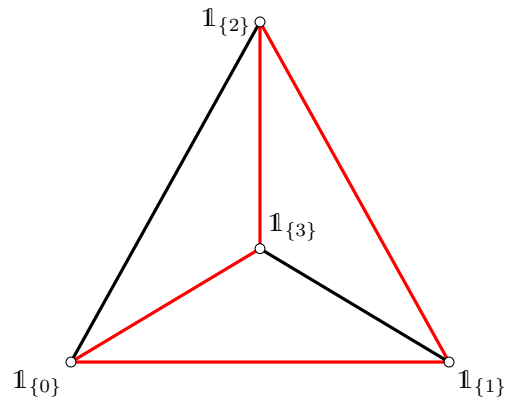
\begin{figure}
\centering
\tdplotsetmaincoords{0}{0}
\begin{tikzpicture}[tdplot_main_coords, scale=5]

% Vértices del tetraedro
\coordinate (A) at (0,0,0);
\coordinate (B) at (1,0,0);
\coordinate (C) at (0.5,0.9,0);
\coordinate (D) at (0.5,0.3,1);

% Aristas coloridas (rojas y negras)
\draw[very thick, red] (A) -- (B);   % (0,1)
\draw[very thick, red] (B) -- (C);   % (1,2)
\draw[very thick, black] (C) -- (A); % (2,0)
\draw[very thick, red] (A) -- (D);   % (0,3)
\draw[very thick, black] (B) -- (D); % (1,3)
\draw[very thick, red] (C) -- (D); % (2,3)

% Vértices
\fill[white] (A) circle (0.013);
\draw (A) circle (0.013) node[below left] {$\mathbb{1}_{\{0\}}$};
\fill[white] (B) circle (0.013);
\draw (B) circle (0.013) node[below right] {$\mathbb{1}_{\{1\}}$};
\fill[white] (C) circle (0.013);
\draw (C) circle (0.013) node[left] {$\mathbb{1}_{\{2\}}$};
\fill[white] (D) circle (0.013);
\draw (D) circle (0.013) node[above right] {$\mathbb{1}_{\{3\}}$};
\end{tikzpicture}

\caption{The probability simplex $\Delta_4$. 
Symmetrization resistant distributions in $\mathbb{Z}_4$ are represented in red (\cref{thm:Z4}).}
\label{fig:Z4 symmres}
\end{figure}

\subsection{Two-point random variables}

\begin{proposition}
For all odd $k \ge 3$ and any distinct $i,j \in [k]$, there exist asymmetric two-point distributions supported on $\{i,j\} \subseteq \mathbb{Z}_k$ which are not symmetrization resistant. 

\begin{proof}
By affine invariance (\cref{lem:affine_invariance}), we may take $p > 1/2$, and we may take our two-point PMF $f = f_k$ on $\mathbb{Z}_k$ to be such that $f_k(-1) = 1-p$ and $f_k(1) = p$ and zero elsewhere. 
We will construct a PMF $g = g_k$ for each $k$ such that $f \convolve g$ is symmetric and 
\begin{equation}\label{eq:entropy limit property} \lim_{p \downarrow \frac{1}{2}} H(g_k) = 0 . \end{equation}

We begin by writing $k = 2l + 1$ and $\delta = \frac{2p-1}{kp-l} = \frac{2p-1}{2lp+p-l}$. 
Define $g(0) = 1-l\delta$; the values of $g$ at the $2l$ remaining elements of $\mathbb{Z}_k = \mathbb{Z}_{2l+1}$ will be $\delta$ ($l$ times) and $0$ ($l$ times), according to the following pattern:
\begin{enumerate}[(i)]
\item If $k \equiv_4 1$, 
 $g_k(1) \coloneqq g_k(2) \coloneqq 0$ 
and $g_k(3) \coloneqq g_k(4) \coloneqq \delta$. 
\item If $k \equiv_4 3$, 
 $g_k(1) \coloneqq \delta$,
 $g_k(2) \coloneqq 0$;
and when $k > 3$, 
 $g_k(3) \coloneqq 0$ and 
 $g_k(4) \coloneqq \delta$. 
\item For all $k \ge 11$, for all integers $r \in [5,l]$, $g_k(r) \coloneqq g_k(r-4)$. 
%(it is easily verified that this condition does not contradict our already-specified values for $g_k(1)$ up to $g_k(5)$ for any $k$). 
\item Finally, for all $r \in [l]$, $g(-r) \coloneqq 0$ iff $g(r)=\delta$ and $g(-r)\coloneqq \delta$ iff $g(r)=0$. 
\end{enumerate}
Intuitively, $g$ consists of one large mass at the origin, and alternating pairs of zeros and $\delta$s on the rest of $\mathbb{Z}_k$. 

Note that $\lim_{p \downarrow \frac{1}{2}} \delta = 0$, so that
$\lim_{p \downarrow \frac{1}{2}} g_k = \mathds{1}_{0}$ and therefore $g_k$ satisfies \eqref{eq:entropy limit property}. 
It remains to show that $g_k$ is a symmetrizer of $f_k$, i.e., that $(g_k \convolve f_k)(r) = (g_k \convolve f_k)(-r)$ for all $r \in [l]$. 

For $r \ge 2$, noting that $g(\pm(r-1)), g(\pm(r+1)) \in \{0, \delta\}$ with $g(-(r-1)) = \delta$ iff $g(r-1) = 0$ iff $g(r+1) = \delta$ iff $g(-(r+1))=0$, 
we have $g(-(r-1)) = g(r+1)$ and $g(-(r+1)) = g(r-1)$, so that
\[
(f \convolve g)(r) 
= pg(r-1) + (1-p)g(r+1)
= pg(-r-1) + (1-p)g(-r+1) 
= (f \convolve g)(-r) .
\]
Finally, noting that the values of $g_k(2) = 0$ and $g_k(-2) = \delta$ for all $k$, we compute
\begin{align*}
(f\convolve g)(1) 
&= pg(0) + (1-p)g(2) \\
%&= pg(0) \\
&= \frac{p^2}{2lp + p -l} \\
&= \frac{p^2 - (1-p)p + (1-p)p}{2lp + p -l} \\
&= \frac{p(2p - 1)}{2lp + p -l} + (1-p)g(0) \\
&= p \delta + (1-p)g(0) \\
%&= p g(-2) + (1-p)g(0) \\
&= (f\convolve g)(-1). \qedhere
\end{align*}
\end{proof}
\end{proposition}

%\subsubsection{Enumerating extreme points}
Returning to the Krein-Milman approach, we would like to find all the extreme points of $\calY(f)$ for $f \in \mathbb{Z}_p$.

One way to find extreme points is to start with the symmetrizers $\hatf$ given in \cref{lem:hatf_are_symmetrizers} and reduce their support one element at a time until they are extreme points. 
However, as the next example will show, this is not sufficient to find all extreme points. 
In particular, for $\mathbb{Z}_p$ with $p$ odd, there are only $\frac{p+1}{2}$ functions $\hatf$. 
In the case of $\mathbb{Z}_5$ these $\hatf$ could be reduced to (at most) $3$ extreme points, but, as we will now see, for some $f$ there are actually $4$.

\begin{example}[$\mathbb{Z}_5$]
We compute the extreme points of an asymmetric two-point $X \sim f$ in $\mathbb{Z}_5$. 
We assume (by affine invariance) that $F \coloneqq \supp(f) = \{\pm1\}$ and $p \coloneqq f(1) > \frac{1}{2}$. 
Take $g = (g_{-2},g_{-1},g_0, g_1, g_2) \in \Ex(\Sym(f))$ (where we have written $g_i \coloneqq g(i)$ for convenient values of $i$) write $G \coloneqq \supp(g)$. 
Note that $\supp(f \convolve g) = F + G$; thus $F + G$ must (at the very least) be symmetric in the sense that $F+G = -(F+G)$. 

Clearly $\abs{G} = 1$ is impossible so we start by examining the case $\abs{G}=2$. 
Writing sets without brackets for convenience, we immediately see that $G = \pm1$ and $G=\pm2$ are only possibilities when $\abs{G}=2$. %$-2,-1$, $-2,0$, $-2,1$, $-2,2$, $-1,0$, $-1,2$, $0,1$, $0,2$, and $1,2$ cannot be the supports of symmetrizers, because the convolution of a variable with these supports would not be symmetric. 
The case $G = \pm2$ would force 
$p-pg_2 = p(1-g_2) = (f \convolve g)(-1) = (f \convolve g)(1) = (1-p)g_2 = g_2-pg_2$, thus $g_2=p$; 
but also
$pg_2 = (f \convolve g)(-2) = (f \convolve g)(2) = (1-p)(1-g_2) = 1-p-g_2+pg_2$, thus $g_2=1-p$ and we conclude $p=\frac{1}{2}$, a contradiction. 
The case $G=\pm1$ produces the unique solution $g = -f$. 
Thus the unique element of $\Ex(\Sym(f))$ with support cardinality 2 is $g=-f$, with support $G=\pm1$. 

Let us now consider the case $\abs{G}=3$; there are $\choose{5}{3}=10$ possibilities. 
Note that $G$ must be minimal; that is, there can be no $s \in \Sym(f)$ with $\supp(s) \subsetneq G$. 
Thus we may immediately disregard the 3 possibilities where $G \supset \pm1$. 
We also find that $G=-2,0,1$ and $G=-1,0,2$ are impossible because $F+G$ would not be symmetric. 
The remaining possibilities are $G=-1,\pm2$, $G=0,\pm2$, $G=1,\pm2$, $G=-2,-1,0$, and $G=0,1,2$. 
Note that it is easily seen that in each case the symmetry equations are linearly independent; thus %(together with the constraint that the probabilities must sum to 1) 
each case admits at most one solution. 

In the case $G=-1,\pm2$, we compute 
\( pg_{-2} = (f \convolve g)(-1) = (f \convolve g)(1) = qg_2 \)
so that $g_{-2} = \frac{q}{p}g_2$. 
Also, 
\( q g_{-1} + p g_2 = (f \convolve g)(-2) = (f \convolve g)(2) = q g_{-2} = \frac{q^2}{p}g_2,  \)
so that 
\[ 0 \le pqg_{-1} = (q^2-p^2) g_2 \le 0 \]
thus $g_{-1} = g_2 = 0$, contradicting our choice of $G$. 

In the case $G=1,\pm2$, we again have $g_{-2} = \frac{q}{p}g_2$, 
and now 
%\[ p g_2 = (f \convolve g)(-2) = (f \convolve g)(2) = q g_{-2} + pg_1 \]
%\[ pg_2 = \frac{q^2}{p}g_2 + pg_1 \]
\[ (p^2 - q^2) g_2 = p^2 g_1 \]
so that $g_1 = \frac{2p-1}{p^2} g_2$ and
\[ 1 = g_2 + g_{-2} + g_1 = g_2\left(1 + \frac{1-p}{p} + \frac{2p-1}{p^2}\right)  = g_2 \left(\frac{p^2 + p-p^2 + 2p-1}{p^2}\right); \]
hence $g_2=\frac{p^2}{3p-1}$ and
$g = \frac{1}{3p-1}(pq, 0, 0, 2p-1, p^2)$. 

In the case $G = 0, \pm2$, the unique solution is $g = \frac{1}{2}(p, 0, 1, 0, q)$ (a solution of the type introduced in \cref{lem:hatf_are_symmetrizers}). 

In the case $G=-2,-1,0$ %we have $qg_{-1} = qg_{-2}$ 
we immediately deduce $g_{-1} = g_{-2}$; the other symmetry equation is 
\( qg_0 + pg_{-2} = pg_0, \)
so that $g_0 = \left(\frac{p}{p-q}\right)g_{-2}$. 
By total probability, $g_{-2} = \frac{2p-1}{5p-2}$, whence 
\( g = \frac{1}{5p-2}(2p-1, 2p-1, p, 0, 0) .\)

In the case $G = 0,1,2$, we immediately conclude $g_1 = g_2$ and the other symmetry equation is $qg_0 = qg_2 + pg_0$. 
Thus $0 \ge (q-p)g_0 = qg_2 \ge 0$, whence $g_0 = g_2 = 0$, contradicting our choice of $G$. 

% general symmetry equations used in the above: 
%\[ q g_0 + pg_{-2} = (f \convolve g)(-1) = (f \convolve g)(1) = qg_2 + pg_0 \]
%and
%\[ q g_{-1} + p g_2 = (f \convolve g)(-2) = (f \convolve g)(2) = q g_{-2} + pg_1 \]

Finally, it is easily checked that any $G$ with $\abs{G} \in \{4,5\}$ would not be minimal; thus there are no more extreme points. 
We conclude that 
\[
\Ex(\Sym(f)) = 
\begin{Bmatrix} 
(0,p,0,q,0), \\
\frac{1}{2}(p, 0, 1, 0, q), \\
\frac{1}{3p-1}(pq, 0, 0, 2p-1, p^2), \\
\frac{1}{5p-2}(2p-1, 2p-1, p, 0, 0) 
\end{Bmatrix}.
\]
Numerically computing the minimum entropy of these extreme points reveals that $f$ is symmetrization resistant for values of $p$ outside the interval $(0.46075,0.53925)$. 
\end{example}

\section{Remarks}
\label{sec:rmk}

Although we did not do so in this paper, it may be of interest to also consider another measure of asymmetry given by 
$$
\tilde{A}_\Phi(\mu) := \inf_{\{\nu\in Sym(\mu)\}} \Phi(\mu*\nu)-\Phi(\mu) .
$$
Note that if $\Phi=V$ and $\mu\in\calP(\RL)$, then $\tilde{A}_V(\mu)=A_V(\mu)$ due to the additivity of variance under convolution. In general, however,
the concavity of $\Phi$ and the fact that convolution can be seen as a limit of convex combinations of shifts implies that 
$\tilde{A}_\Phi(\mu)\leq A_\Phi(\mu)$. For $\Phi=H$, we may write $\tilde{A}_H(X)=\inf_{\{Y\in Sym(X)\}} H(X+Y)-H(X)$,
which has an interpretation as a minimization of mutual information.

We note that the optimization problems arising in the definition of both ${A}_\Phi$ and $\tilde{A}_\Phi$ are reminiscent of important optimization problems that arise in probability and information theory; these are of the form:
``Given the distribution $\mu\in \calP(G)$ of a random variable $X$ taking values in $G$, what is the distribution $\nu$ within a relevant feasible set such that the convolution $\mu\convolve \nu$ minimizes some objective function $\Phi:\calP(G)\ra\RL$?''
A classical example of such optimization problems arises in the context of information and communication theory.
Indeed, C. Shannon \cite{Sha48} showed that the capacity of a point-to-point communication channel with signals taking values in a discrete group $G$ and with additive noise $Y$ can be determined as
$$
C=\sup_{X} \; [ H(X+Y)-H(Y)] .
$$
Moreover, in the setting of real-valued signals, he showed that 
$$
C=\sup_{X : V(X) \le s} \; [ h(X+Y)-h(Y)] ,
$$
where $s$ is the constraint on the signal power, and $h$ is the differential entropy (defined for probability density functions $f$ on $\RL$ by the integral $h(f)=-\int_{\RL} f(x)\log f(x) dx$ with respect to Lebesgue measure). 
Since Gaussian distributions uniquely maximize the entropy subject to a variance constraint, and since the relative entropy from the class of Gaussian distributions is precisely the difference of entropies, the problem of determining the capacity reduces to
finding $Y$ (subject to a variance constraint) independent of $X$ 
such that the relative entropy of $X+Y$ from Gaussian is minimized. 
%The initial motivation of \cite{KMSVV99} was \emph{Gaussianization} of $X+Y$: that is, given $X$, is it possible to find $Y$ such that $X+Y$ is Gaussian? 
%The requirement that $X+Y$ be symmetric is a relaxation of the Gaussian constraint. 
%Indeed, though this was not noticed by  \cite{KMSVV99}, the Gaussianization problem is equivalent 
%to the question of determining the capacity-achieving distribution for an additive noise channel with noise $X$, which is known to be a challenging problem in information theory. 
%Indeed, the channel capacity for the channel $Y \mapsto X+Y$ is the supremum of the mutual information $I(X+Y;Y) = H(X+Y) - H(X + Y \mid Y) = H(X+Y) - H(X)$. 
%Since the distribution of the noise $X$ is fixed, maximizing the mutual information $I(X+Y;Y)$ is equivalent to maximizing $H(X+Y)$. 
%Since $X$ and $Y$ are independent, $\Var(X+Y)=\Var(X)+\Var(Y)$, so the transmit power variance constraint $\Var(Y) \le c$ induces an equivalent variance constraint on $X+Y$. 
%Thus, the problem is equivalent to maximizing the entropy $H(X+Y)$ subject to a constraint $\Var(X+Y) \le c$; 
%thus the capacity-achieving distribution $Y$ is one that results in Gaussian $X+Y$. 
See, e.g., \cite{ENT18:1, MNT19:isit,  MNT21} for some recent developments in connection with this question, which was the original motivation of \cite{KMSVV99} (in the sense that symmetry is a relaxation of Gaussianity). 
% While this paper focused exclusively on  symmetrization, 
% we shall not address these questions any further, instead 

Exploration of symmetrization resistance for nonabelian groups would be particularly interesting, and one expects the group structure to play a major role. By Cayley's theorem, determining  the symmetrization resistant distributions on symmetric groups of all finite orders would yield the same on arbitrary finite groups. However, even the determination on  symmetrization resistant distributions on symmetric groups of small orders seems nontrivial.

\bibliographystyle{plain} 
\bibliography{pustak,extra} %bibliography

\end{document}